\documentclass{article}
\usepackage[scale=0.65]{geometry}
\usepackage{graphicx, nicefrac}

\usepackage{amsthm, amssymb, amsmath, dsfont, mathrsfs, bm, color, varwidth, upgreek, float}
\newtheorem{theorem}{Theorem}[section]
\newtheorem{definition}[theorem]{Definition}
\newtheorem{remark}[theorem]{Remark}
\newtheorem{lemma}[theorem]{Lemma}
\newtheorem{proposition}[theorem]{Proposition}
\newtheorem{corollary}[theorem]{Corollary}
\newtheorem*{theorem*}{Theorem}
\theoremstyle{definition}
\newtheorem*{definition*}{Definition}

\DeclareMathOperator{\e}{\mathbb{E}}
\DeclareMathOperator{\p}{\mathbb{P}}

\usepackage[affil-sl]{authblk}
\usepackage{orcidlink}
\title{Competitive displacement in\\ the grass-bushes-trees process on a scale-free network}

\author[ ]{John Fernley}
\author[ ]{Daniel Valesin}
\affil[ ]{University of Warwick Department of Statistics}
\affil[ ]{\rm \{john.fernley, daniel.valesin\} @warwick.ac.uk}
\date{\today}

\usepackage{todonotes}

\usepackage[maxnames=6, backend=bibtex, url=false, isbn=false, doi=false, eprint=false]{biblatex}
\renewbibmacro{in:}{}
\begin{document}

\maketitle

\begin{abstract}
\noindent
	The grass-bushes-trees process is a two-type contact process in which the types face asymmetric birth mechanisms. Vertices of a graph can either be occupied by a tree, occupied by a bush, or empty (containing grass). The trees die with rate~1 and give birth with rate~$\lambda_{\mathcal T}$ on any neighbouring vertex (possibly taking the place of a bush). Bushes die with rate~1 and give birth with rate~$\lambda_{\mathcal B}$ on any empty neighbouring vertex; in particular, they are forbidden from taking the place of trees. We study this process on a power law random graph, namely the erased configuration model.  We study which graph parameters lead to the possibility of coexistence versus the necessity of competitive displacement, i.e. joint metastability or fast extinction of the bushes.

\vspace{0.5em}

{\scriptsize

}
\end{abstract}

\section{Introduction and results}

The contact process, or SIS model, is a simple Markov model for an infection on a graph. The two mechanisms are that each infected vertex recovers at rate $1$, and each edge between an infected and a healthy vertex communicates the infection at rate $\lambda>0$. This process can be built from the \emph{graphical construction} \cite{MR2108619}, in which an independent Poisson process is attached to each vertex and to each directed edge, and the infection through time is then a deterministic function of the arrival marks in those processes.

The grass-bushes-trees process is a simple multitype contact process, building on these dynamics to model competing spreading. 
This two-type contact process is constructed from two independent contact process graphical constructions of rates~$\lambda_{\mathcal B}$ and~$\lambda_{\mathcal T}$. The contact process of type ${\mathcal T}$, the trees, runs independently, while the contact process of type ${\mathcal B}$, the bushes, is instantaneously removed from any site occupied by a tree. All previous works \cite{gbt_interface,gbt_Zd_worse,gbt_Zd_better,andjel2010tightness} have considered this process on a lattice.

The name suggests that this removal is caused by the trees depriving smaller plants of light, but it might also model competition between viruses, see the conclusion of \cite{du2022viral} for a range of empirical evidence in this area. The \emph{competitive exclusion principle} \cite{volterra} is the idea that one of two species competing for the same resource will out-compete the other. Through the last hundred years of biology research, this principle has been applied in various contexts and shown to not hold in others. With regards to viruses, in \cite{nickbakhsh2019virus} they examine the reason that ``common cold infections are less frequent during flu seasons'' so there is an empirical interaction here to explain.
Mean-field compartmental models of the interaction between viral infections were studied in \cite{MR3147742,lazebnik2022generic} and simulation-based approaches to the same question are in \cite{pinotti2020interplay, vafadar2021competitive}.

For the purposes of social modelling, foremost the virus context, we consider this process on a scale-free network model. Social graphs are thought to have a power-law tail in their degree distribution, so we will study the configuration model graph $G_n$ on $[n]=\{1,\dots,n\}$ with i.i.d. stub-counts drawn from a distribution with mass function $p$, satisfying:
\begin{equation}\label{eq_condition_stubcount}
	p(\{n:n \ge 3\})=1 \qquad \text{and}\qquad \lim_{n \to \infty} \frac{p(n)}{\beta n^{-\tau}} = 1.
\end{equation}

This means that every vertex $i\in[n]$ gets some number $D_i$ of ``stubs'' from which to construct the configuration model, and we condition that $\sum D_i$ is even so that every stub can be paired. The total number of stubs grows linearly when $\tau>2$, but we will also consider $\tau \in (1,2)$ where it is of order $n^{{1}/{(\tau-1)}}$ in probability.

We then move to the \emph{erased configuration model} version of the graph  \cite{remco}, which is the largest simple graph contained in the multigraph. Equivalently this could be thought of as infection processes on the full multigraph, but with infection dynamics where each vertex infects each of its unique neighbours at constant rate.
The erased configuration model in expectation has $\Theta(n \log n)$ edges for $\tau\in(1,2]$, so it is closer to a sparse network than the multigraph version.

The study of the contact process on scale-free networks began as a response to the physics literature: first the paper of \cite{PhysRevLett.86.3200} had a prediction that the metastable epidemic density in a preferential attachment graph would be $e^{-\Theta(1/\lambda)}$,  which was corrected to $\lambda^{\Theta(1)}$ by \cite{MR2298278}. This article is about rank one network models, and their investigation came later, with another incorrect prediction of no epidemics at small $\lambda$ values when $\tau>3$ by \cite{pastor2001epidemic},
which was corrected by \cite{chatterjee2009contact} who find a positive probability of a stretched-exponential epidemic period for all $\tau>2$ and $\lambda>0$. Later, this was strengthened to an exponential epidemic period by \cite{mountfordvalesinyao}.
This controversial history reveals the difficulty in accurate simulations of the contact process (and so also a reason to expect some small-$n$ effects in our Figure \ref{fig_sim}).


Let~$G_n$ be the erased configuration model with degree distribution~$p$ satisfying~\eqref{eq_condition_stubcount} with~$\tau > 1$. In our first result, we rule out coexistence completely when $\tau <2$.
\begin{theorem}\label{thm_main_heavy}
	Assume that~$\tau \in (1,2)$.
	 Then, there exist~$C > 0$ and a sequence of positive real numbers~$(\eta_n)_{n \ge 1}$ with~$\eta_n \to 0$ such that the following holds. With high probability as~$n \to \infty$,~$G_n$ is such that the grass-bushes-trees process~$(\mathcal T_t,\mathcal B_t)_{t \ge 0}$ on~$G_n$ with parameters~$\lambda_{\mathcal B}$ and~$\lambda_{\mathcal T}$ satisfies
	\begin{equation}\label{eq_no_coex}
		\sup\mathbb P(\mathcal T_{C \log n} \neq \varnothing,\; \mathcal B_{C \log n} \neq \varnothing) < \eta_n,
	\end{equation}
	where the supremum is taken over all initial configurations~$(\mathcal T_0,\mathcal B_0)$.
\end{theorem}

The following two theorems divide the lighter-tail behaviour. First, we show that when $\tau\geq 2$, any $\lambda_{\mathcal T}$ has some set of $\lambda_{\mathcal B}$ where we see competitive displacement.

\begin{theorem}\label{thm_no_coex_lambdas}
	Assume that~$\lambda_{\mathcal B} \le \lambda_{\mathcal T}$.
	 Then, there exist~$k > 0$ and a sequence of positive real numbers~$(\eta_n)_{n \ge 1}$ with~$\eta_n \to 0$ such that the following holds. With high probability as~$n \to \infty$,~$G_n$ is such that the grass-bushes-trees process~$(\mathcal T_t,\mathcal B_t)_{t \ge 0}$ on~$G_n$ with parameters~$\lambda_{\mathcal B}$ and~$\lambda_{\mathcal T}$ satisfies
	\begin{equation}\label{eq_no_coex2}
		\sup\mathbb P(\mathcal T_{n^k} \neq \varnothing,\; \mathcal B_{n^k} \neq \varnothing) < \eta_n,
	\end{equation}
	where the supremum is taken over all initial configurations~$(\mathcal T_0,\mathcal B_0)$.
\end{theorem}

\begin{remark}\label{remark_tree_survival}
	Looking at the bounds~\eqref{eq_no_coex} and~\eqref{eq_no_coex2}, letting~$t_n$ be either~$C\log n$ (in the case of~\eqref{eq_no_coex}) or~$n^k$ (in case of~\eqref{eq_no_coex2}), one could imagine that the reason that~$\mathbb P(\mathcal T_{t_n} \neq \varnothing,\; \mathcal B_{t_n} \neq \varnothing)$ is low is that~$\mathbb P(\mathcal T_{t_n} = \mathcal B_{t_n}=\varnothing)$ is close to~1. This is however not the case, at least for certain choices of~$(\mathcal T_0,\mathcal B_0)$. For instance, if~$v^*$ is a vertex of high degree, say larger than~$n^a$ for some~$a > 0$, and~$(\mathcal T_0,\mathcal B_0) =(\{v^*\}, G_n \setminus \{v^*\})$, then the trees are likely to survive for a long time (at least~$\exp(c_\lambda n^a)$, for some~$c_\lambda > 0$), and in case they do,~\eqref{eq_no_coex} guarantees that the bushes are likely all gone by time~$t_n$.
\end{remark}

Our final theorem is the coexistence case: in this lighter tail region $\tau>2$, we can have coexistence when $\lambda_{\mathcal B}$ is large, $\lambda_{\mathcal T}$ is small, and the initial conditions are stable.

\begin{theorem}\label{thm_coexistence}
	Assume that~$\tau>2$ and fix some $\gamma\in (0,1)$. Then, if we take  $\lambda_{\mathcal T}, \epsilon>0$ sufficiently small and $\lambda_{\mathcal B}$ sufficiently large (independently for all three), there exists~$c>0$ and a sequence of positive real numbers~$(\eta_n)_{n \ge 1}$ with~$\eta_n \to 0$ such that the following holds. With high probability as~$n \to \infty$,~$G_n$ is such that the grass-bushes-trees process~$(\mathcal T_t,\mathcal B_t)_{t \ge 0}$ on~$G_n$ with parameters~$\lambda_{\mathcal B}$ and~$\lambda_{\mathcal T}$ satisfies
	\begin{equation*}
		\inf \mathbb P(\mathcal{B}_{\exp(cn)}\neq \varnothing) > 1-\eta_n,
	\end{equation*}
	where the infimum is taken over all configurations~$(\mathcal T_0,\mathcal B_0)$ with~$|\mathcal T_0| \le \epsilon n$ and~$|\mathcal B_0| \ge \gamma n$.
\end{theorem}

Looking at the trees in isolation, we are talking simply about a contact process on a scale-free network and thus we believe that the results of \cite{mountfordvalesinyao,schapira15} extend to the erased configuration model: that the trees will survive with positive probability for some exponential time in all three theorems. In the $\tau> 2$ case \cite{mountfordvalesinyao} we expect that the proofs of that work would transfer to the erased configuration model, which has the same local limit and globally only differs by $o(n)$ edges. Note however that in the $\tau\in(1,2]$ case \cite{schapira15} erasing actually removes the majority of the edges, so we cannot claim that those proofs transfer: this is why we include Remark \ref{remark_tree_survival} to at least guarantee tree survival for a stretched exponential period.

\begin{figure}[ht]
\centering
\includegraphics[width=0.85\textwidth]{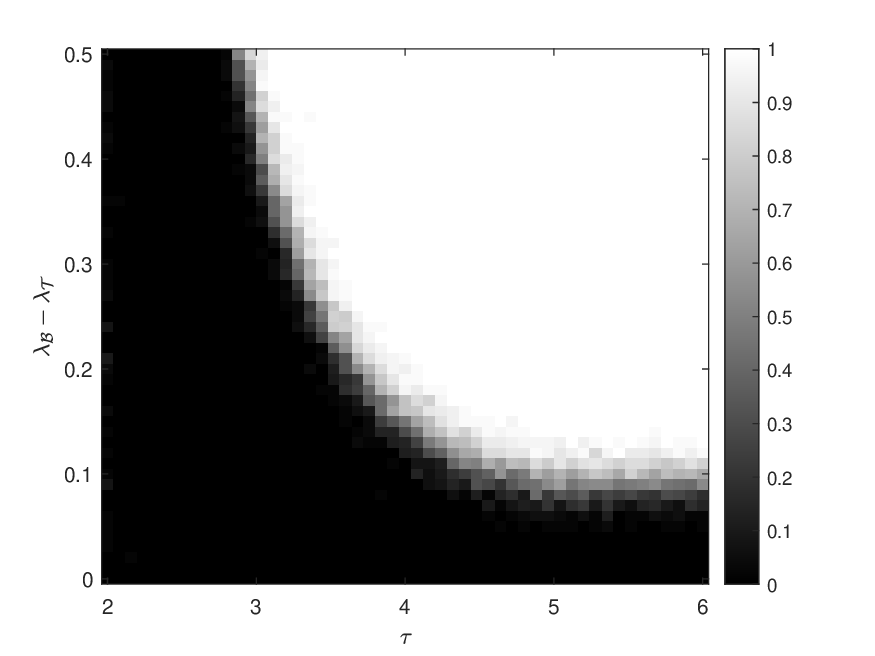}
\caption{A simulation on configuration models with $n=5000$ showing the proportion of $50$ samples in which $100$ bushes remain after $10^5$ infections or recoveries. Initially each vertex is independently occupied by either a bush or a tree with equal probability, and the degrees are drawn from a rounded Pareto distribution with minimum $3$. In all cases we have $\lambda_{\mathcal T}=0.3$ and vary $\lambda_{\mathcal B}\geq\lambda_{\mathcal T}$. Note this figure illustrates all three theorems: the two competitive displacement regimes are the continuing black regions underneath and on the left of the borders, and the coexistence regime is the white we find by going arbitrarily high.}\label{fig_sim}
\end{figure}

The transition at $\tau=2$ is not about infinite mean degree, as in the local limit the degrees already have infinite expectation for $\tau\leq 3$. Rather, we should think in the scenario of  Theorem \ref{thm_main_heavy} that a small stationary infection level is concentrated in the high degree vertices, and $\tau=2$ is the changepoint where the network becomes \emph{vulnerable to attack} in the sense of \cite{MR2076725}. This is relevant because when the network is attack-vulnerable most vertices can only communicate through the large degree vertices, and those are typically occupied by a tree. For a rank one scale-free network like our model, vulnerability to attack coincides with non-sparsity.

Our results should be compared with \cite{MR2537523}, who study a model where instead both types can give birth only onto vacant sites, so neither can displace the other. On the regular tree, they demonstrate \emph{local} coexistence when both contact process parameters are in the weak survival window, and \emph{local} competitive displacement when one of them is larger than the other and in the strong survival phase. This resembles our $\tau>2$ dichotomy although it is a result really about the ergodic behaviour of infinite systems, and of course our coexistence result has to work against the less favourable interaction of trees removing bushes.

An even closer work, \cite{MR4706682} look at two competing first-passage percolation processes on a scale free network with $\tau\in(1,2)$. Their result specialises to the same erased configuration model graph that we consider, and if we make all edge times exponentially distributed it would be essentially an SI version of our SIS, with again the constraint that both types can give birth only onto vacant sites, and the result is then that for any infection rates one of the types occupies all but one vertex (with high probability).
Our results provide an analogous fast extinction in a system with both recovery and reinfection, where the dominant type is metastable rather than fixed at every vertex but one. Moreover we demonstrate when coexistence behaviours are also possible, through the whole range $\tau>1$ of power-law tails.

The proof of Theorem \ref{thm_main_heavy} requires a detailed structural understanding of the network, which we defer to the final section. The most important part of the description is that we give each vertex an iid stub-count and then divide the network into a \emph{core} of the $\lfloor n^{1-\epsilon} \rfloor$ largest stub-counts and the \emph{satellites} are the remaining $n-\lfloor n^{1-\epsilon} \rfloor$. After this division, the components induced by the satellites are uniformly bounded by a constant.

To exploit this, we construct a weighted supermartingale tracking the bush-occupancy of each core vertex and each induced component in the satellites (i.e. each satellite). Because we must control arbitrarily large bush infection rates $\lambda_{\mathcal{B}}$, we have none of the usual criteria for \emph{subcriticality} of the bushes on the satellites but we have a novel remedy for this by looking into the future of the graphical construction at each satellite to determine its weight.

Reinforcing our previous comment on attack vulnerability, it is enough in this construction to stochastically lower bound the trees by an auxiliary \emph{bonsai process}, consisting of independent two-state chains on the core. These bonsais do not interfere at all in $n-\lfloor n^{1-\epsilon} \rfloor$ of the vertices, but they occupy the core enough to obtain the claimed bush extinction in logarithmic time. In the core, we use our high probability graph construction to guarantee that every vertex has a certain number of tree neighbours at any given time. This lower bound is then the birth rate of bonsais at that vertex.

The other parameter regime where we can rule out coexistence is $\lambda_{\mathcal B} \le \lambda_{\mathcal T}$. This is a shorter argument greatly exploiting the \emph{self-duality} of the contact processes: that paths running backwards in time have the same law. Using this we can explore the history of a putative bush vertex at the terminal time to discover a tree in its infection history, and so discover that the vertex was in fact a tree. Exploring every vertex in this way gives in general a polynomial bound on the time until the bushes are extinct, but this could be strengthened to $C\log^2 n$ under an additional assumption like $\lambda_{\mathcal B}\geq 2$ which would make coalescence easier.

In the rest of the article, first in Section \ref{sec_prelim} we will cover some background on the infection and graph models under consideration, then in Section \ref{s_no_coex} we will construct the above supermartingale to rule out coexistence when $\tau<2$. For the lighter tails, in Section \ref{s_second_no_coex} we have the above duality argument and in Section \ref{s_coex} we show coexistence from the literature bounds on the edge-boundaries of the graph. Finally, Section \ref{s_graph_properties} collects proofs of the required graph properties through all three cases.

\section{Preliminaries on the configuration model and the contact process}\label{sec_prelim}

First, we introduce the basic properties of the random graph model that we consider, and of the contact process which is the exact model of the trees, and of the bushes in the absence of trees.

\subsection{Configuration model}

Given a sequence of positive integers $(d_1,\ldots,d_n)$ with even sum, the \emph{configuration model} is the random multigraph on $[n]$ constructed as follows. Attach $d_i$ distinct half-edges, or \emph{stubs}, to each vertex $i$, and choose a perfect matching of all stubs uniformly at random. Each matched pair forms an edge between the vertices to which its stubs are attached; self-loops and parallel edges are allowed. Equivalently, the matching can be revealed by sequentially pairing an unmatched stub with another uniform random unmatched stub.

In this paper, we sample $D_1,\ldots,D_n$ independently with law $p$ satisfying \eqref{eq_condition_stubcount}, condition on $\sum_{i=1}^n D_i$ being even, and use these as the stub-counts in the construction above. The graph $G_n$ is then obtained by \emph{erasing} the multigraph to a simple graph, deleting all self-loops and replacing sets of parallel edges by a single edge. Since $p$ assigns positive probability to both parities, the even-sum condition has probability tending to $1/2$, so the conditioning preserves events that hold with high probability for the independent stub-counts.

We will need to use several properties of the random graph~$G_n$ which hold with high probability as~$n \to \infty$. We state these properties in the sections about the grass-bushes-trees process, but postpone the proofs to Section~\ref{s_graph_properties}. For now, we only state the following, which follows from~\cite[Theorem 4.25]{remco}.
\begin{lemma}\label{lem_connectedness}
For all~$\tau > 1$, with probability tending to~$1$ as~$n \to \infty$, the random graph~$G_n$ is connected.
\end{lemma}

\subsection{Contact process}

For the contact process we need first a result on the recovery time for an epidemic on a general graph, and on the sensitivity of this recovery time to the addition of a vertex.

\begin{lemma}\label{lem_prelim_moment}
	Let~$G=(V,E)$ be a finite graph and let~$H$ be a graphical construction for a contact process with rate~$\lambda$ on~$G$. For~$B \subseteq V$, let~$T^B:=\inf\left\{ t: \bm{\xi}^B_t = \varnothing\right\}$. Then, the following statements hold.
	\begin{itemize}
		\item[$(\mathrm{a})$] We have
			\begin{equation*}
				\mathbb E\left[ \exp\left\{\alpha T^V\right\}\right] < \infty
			\end{equation*}
			for any positive~$\alpha < -\log\left( 1- e^{-\lambda |E|}(1-e^{-1})^{|V|}\right)$.
		\item[$(\mathrm b)$] Given~$A \subset V$ and~$v \in V \setminus A$ and a non-negative, non-decreasing function~$f:\mathbb R \to\mathbb R$, we have
			\[
				\mathbb E\left[\left.f\big(T^{A \cup \{v\}}\big) - f\big( T^A\big)\;\right| \; T^A\right] \le \mathbb E\big[ f\big(T^{\{v\}}\big)\big].
			\]
	\end{itemize}
\end{lemma}
\begin{proof}
		For~$(\mathrm{a})$, let~$N$ be the smallest~$n\in \mathbb N$ for which, in the time interval~$[n-1,n]$, there is a recovery mark at every vertex of~$G$, and no birth arrows in any edge of~$G$. Then,~$T^V \le N$ and~$N \sim \mathrm{Geom}(e^{-\lambda |E|}(1-e^{-1})^{|V|})$, from which the result follows.

		For~$(\mathrm{b})$, enlarging the probability space, we let~$H'$ be another graphical construction for the contact process with rate~$\lambda$ on~$G$, independent of~$H$. It is then possible to define an auxiliary process~$(\bm{\zeta}_t)_{t \ge 0}$ with the following properties:
			\begin{itemize}
				\item[$\mathrm{(i)}$] $(\bm{\zeta}_t)_t$ is a contact process started from~$\mathds{1}_{\{v\}}$;
				\item[$\mathrm{(ii)}$] we have~$\bm{\xi}^{\{v\}}_t \setminus \bm{\xi}^A_t \leq \bm{\zeta}_t$ for all~$t$;
				\item[$\mathrm{(iii)}$] $(\bm{\zeta}_t)_t$ and~$(\bm{\xi}^A_t)_t$ are independent. 
			\end{itemize}
			We explain the definition of~$(\bm{\zeta}_t)_t$ in words, and leave the verification of these properties to the reader. We set~$\bm{\zeta}_0 = \mathds{1}_{\{v\}}$.  Letting~$\Xi_t:=\bm{\xi}^{\{v\}}_t \setminus \bm{\xi}^A_t$, the evolution of~$(\bm{\zeta}_t)_t$ obeys, at every time~$t$:
	\begin{itemize}
		\item from~$H$: the recovery marks at~$\Xi_t$, and the birth arrows that start from~$\Xi_t$;
		\item from~$H'$: the recovery marks everywhere else, and the birth arrows that start from everywhere else.
	\end{itemize}

	Now, using~$T^{A \cup \{v\}} = \max\{T^A,T^{\{v\}}\}$ and monotonicity of~$f$, we can bound
	\begin{align}\label{eq_aux_alpha}
		f\big( T^{A \cup \{v\}}\big)- f\big( T^A\big) \le \mathds{1}\{T^{\{v\}} > T^A \}\cdot f\big( T^{\{v\}} \big).
	\end{align}
	On the event~$\{T^{\{v\}} > T^A\}$, for all~$t > T^A$, we have~$\bm{\xi}^{\{v\}}_t = \bm{\xi}^{\{v\}}_t \setminus \bm{\xi}_t^A  \subseteq \bm{\zeta}_t$, where the inclusion follows from~$\mathrm{(ii)}$.  Hence,~$T^{\{v\}}\le \inf\{t: \bm{\zeta}_t = \varnothing\} $ on~$\{T^{\{v\}} > T^A\}$. Using this with~\eqref{eq_aux_alpha}, dropping the indicator function on the right-hand side and taking~$\mathbb E[\cdot \mid T^A]$, we obtain
	\begin{align*}
		\mathbb E\left[\left.f \big( T^{A \cup \{v\}}\big) - f\big( T^A \big) \;\right| \; T^A\right] &\le \mathbb E\left[\left. f\big(\inf\{t: \bm{\zeta}_t = \varnothing \}\big)  \;\right|\; T^A\right]\\
		&\stackrel{\mathrm{(iii)}}{=}\mathbb E\left[ f\big(\inf\{t: \bm{\zeta}_t = \varnothing \}\big)  \right] \stackrel{\mathrm{(i)}}{=} \mathbb E\big[f\big( T^{\{v\}} \big)\big].\;\;\;\qedhere
	\end{align*}
\end{proof}


The other thing we need to do in this section is gather some results about the contact process on star graphs. These are  known (see \cite{valesin2024contact}) but we include a sketch proof for the version presented here.

\begin{definition}
Let~$\lambda > 0$.
Let~$S$ be a star graph with~$n$ leaves and center~$o$, and let~$\xi \in \{0,1\}^S$. We say that~$S$ is:
\begin{itemize}
\item \emph{lit} in~$\xi$ if~$|\{x: \xi(x)=1\}| \ge \frac{\lambda \wedge 1}{8e} n$;
\item \emph{semi-lit} in~$\xi$ if~$|\{x: \xi(x)=1\}| \ge \frac{\lambda \wedge 1}{16e^2} n$;
\item \emph{weakly lit} in~$\xi$ if~$|\{x: \xi(x)=1\}| \ge \frac{\lambda \wedge 1}{32e^3} n$
\end{itemize}
(although these notions depend on~$\lambda$, we omit this dependence from the terminology).
\end{definition}

\begin{lemma}\label{lem_all_stars}
	There exists a constant~$c_{\mathrm{star}}>0$ such that the following holds for all~$\lambda > 0$ and all~$n$ large enough (depending on~$\lambda$). Let~$S$ be a star graph with~$n$ leaves and center~$o$. Let~$(\bm{\xi}_t)_{t \ge 0}$ be the contact process with rate~$\lambda$ on~$S$ started from a deterministic configuration~$\bm{\xi}_0$. Let~$T:=\exp\{c_{\mathrm{star}}(\lambda^2 \wedge 1) n\}$. Then, the following statements hold.
\begin{itemize}
\item[$\mathrm{(a)}$] If~$\bm{\xi}_0$ is not identically~0, then
	\[
		\mathbb P\left(S \text{ is lit in }\bm{\xi}_1 \right) > c_{\mathrm{star}}\cdot  (\lambda \wedge 1).
	\]
\item[$\mathrm{(b)}$]  If~$S$ is lit in~$\bm{\xi}_0$, then
	\[
		\mathbb P \left(S \text{ is semi-lit in } \bm{\xi}_t \text{ for all } t \le T,\; \int_t^{t+1} \bm{\xi}_s(o)\mathrm{d}s \geq \frac12 \text{ for all } t \in [0,T] \right) > 1 - T^{-1}.
	\]
\item[$\mathrm{(c)}$]  If~$S$ is semi-lit in~$\bm{\xi}_0$, then
	\[
		\mathbb P \left(S \text{ is weakly lit in } \bm{\xi}_t \text{ for all }t \le 1 \text{ and is lit in } \bm{\xi}_1 \right) > 1 - T^{-1}.
	\]
\end{itemize}
\end{lemma}
\begin{proof}
By monotonicity, it suffices to consider $\lambda \leq 1$.
For $\mathrm{(a)}$, with probability at least $c\lambda$ the centre becomes infected by time $1/2$ and has no recovery before time $1$; a binomial bound on infections of the other leaves during $[1/2,1]$ then gives the claim.

For $\mathrm{(b)}$ and $\mathrm{(c)}$, use \cite[Lemma 2.20]{valesin2024contact} and its proof. If the star is initially lit, a Chernoff bound shows that enough initially infected leaves have no recovery during $[0,1]$ to keep the star semi-lit throughout that interval. If the star is initially semi-lit, the same argument keeps it weakly lit throughout the interval. In both cases, the lemma gives the lit threshold at time $1$, and these conclusions hold except with probability $e^{-c\lambda^2n}$. This proves $\mathrm{(c)}$. For $\mathrm{(b)}$, \cite[Lemma 2.19]{valesin2024contact} also bounds the centre's vacant time during $[0,1]$ by $1/4$, except with probability $e^{-c\lambda^2n}$. Iterating these one-step estimates through time $\lceil T\rceil+1$ and taking a union bound, after reducing $c_{\mathrm{star}}$, shows that with probability greater than $1-T^{-1}$ the star stays semi-lit through time $T$ and its centre is occupied for at least $3/4$ of every interval $[k,k+1]$, $k=0,\ldots,\lceil T\rceil$. An arbitrary unit interval overlaps at most two such intervals, giving the claimed $1/2$.
\end{proof}

%

\section{No coexistence: the case~$\tau \in (1,2)$}\label{s_no_coex}
In this section, we prove Theorem~\ref{thm_main_heavy} which forbids coexistence for any infection rates and for any initial conditions, when~$\tau \in (1,2)$.
The assumption~$\tau \in (1,2)$ is in force throughout.

\subsection{Graph conditions}

Our proof of quick extinction of the bushes will rely on key properties of the random graph~$G_n$. In order to be able to introduce these properties, we start by stating five lemmas concerning the behaviour of~$G_n$ as~$n \to \infty$. While the proof of the first lemma is taken directly from the literature, the proofs of the remaining four will require some work. We postpone these proofs to Section~\ref{s_proof_graph_12}.

\begin{lemma}\label{lem_event_Bn}
	Defining
\begin{equation}\label{eq_def_event_Bn}
	\mathcal B_n := \left\{\begin{array}{l}\text{the set of edges of $G_n$ contains} \\
		\left\{\{i, j\}: i,j \in [n], \;i \neq j,\; ij \le 4n(\log n)^{-10}\right\}
	\end{array}\right\},
\end{equation}
	we have
	\begin{equation*}
		\mathbb P(\mathcal B_n) \xrightarrow{n \to \infty} 1.
	\end{equation*}
\end{lemma}

\begin{lemma}\label{lem_bounds_degrees}
There exists~$\mathsf{k} > 0$ such that, with high probability as~$n \to \infty$, we have
	\[
		\max\left\{1,\; {(\log n)^{-10}}\cdot \frac{n}{i}\right\} \le \deg(i) \le (\log n)^{\mathsf{k}}\cdot \frac{n}{i} \quad \text{for all } i \in [n].
	\]
\end{lemma}

\begin{lemma}\label{lem_degrees_into_sats}
	There exists~$\mathsf{k}' > 0$ such that for any~$\epsilon \in (0,1)$, with high probability as~$n \to \infty$, we have
	\[
		|\{j: n^{1-\epsilon} \le j \le n,\; i \sim j \}| \le \log n+ (\log n)^{\mathsf{k}'} \cdot \left(\frac{n^\epsilon}{i}\right)^{\frac{2-\tau}{\tau-1}} \cdot \frac{n}{i}\quad \forall i \in [n],\; i \ge n^\epsilon.
	\]
\end{lemma}

\begin{lemma}\label{lem_sizes_satellites}
	For~$\epsilon >0$ small enough, there exists~$m \in \mathbb N$ such that, with high probability as~$n \to \infty$,  in the subgraph of~$G_n$ induced by the set of vertices~$\{i\in [n]: i>n^{1-\epsilon}\}$, every connected component has at most~$m$ vertices.
\end{lemma}

We now introduce a notion of ``good graphs'', which are essentially graphs that satisfy the properties guaranteed by the five lemmas above. With this definition at hand, we will be able to switch to deterministic (and good) graphs when studying the grass-bushes-trees process, thus eliminating the need to handle the graph randomness.

The constants~$\mathsf{k}$ and~$\mathsf{k}'$ that appear in the definition below are those from Lemma~\ref{lem_bounds_degrees} and Lemma~\ref{lem_degrees_into_sats}, respectively.

\begin{definition}
Let~$\epsilon > 0$ and~$m \in \mathbb N$.
We say that a graph~$G$ with vertex set~$V=\{1,\ldots, n\}$ is~\emph{$(\epsilon,m)$-good} if the following conditions hold. First,~$G$ is connected. Second,
\begin{equation}\label{eq_guaranteed_edges}
	\text{every two distinct vertices~$i,j$ with~$ij \le \frac{4n}{(\log n)^{10}}$ are neighbours}.
\end{equation}
Third,
	\begin{equation}\label{eq_bounds_degrees}
		\max\left\{1,\; {(\log n)^{-10}}\cdot \frac{n}{i}\right\} \le \deg(i) \le (\log n)^{\mathsf{k}}\cdot \frac{n}{i} \quad \text{for all } i \in [n].
\end{equation}
Fourth,
	\begin{equation}\label{eq_neighbours_in_sat}
		|\{j: n^{1-\epsilon} \le j \le n,\; i \sim j \}| \le \log n+ (\log n)^{\mathsf{k}'} \cdot \left(\frac{n^\epsilon}{i}\right)^{\frac{2-\tau}{\tau-1}} \cdot \frac{n}{i}\quad \forall i \in [n],\; i \ge n^\epsilon.
\end{equation}
Fifth,
\begin{equation}\label{eq_small_satellites}
\begin{array}{c}
\text{the subgraph induced by $\{i: i > n^{1-\epsilon}\}$ has connected}\\
\text{components with at most $m$ vertices.}
\end{array}
\end{equation}
We call the set of vertices~$\{i: i \le n^{1-\epsilon}\}$ the \emph{core}, and each connected component of the subgraph induced by~$\{i: i > n^{1-\epsilon}\}$ is called a \emph{satellite}.
\end{definition}

The following is a direct consequence of Lemmas~\ref{lem_connectedness},~\ref{lem_event_Bn},~\ref{lem_bounds_degrees},~\ref{lem_degrees_into_sats}, and~\ref{lem_sizes_satellites}.
\begin{proposition}\label{prop_graph_is_good}
	There exist~$\epsilon \in (0,1)$ (which can be taken as small as desired) and~$m \in \mathbb N$ (depending on~$\epsilon$) such that
\[
\mathbb P(G_n \text{ is $(\epsilon,m)$-good}) \xrightarrow{n \to \infty} 1.
\]
\end{proposition}

\begin{lemma}\label{lem_core_gbt}
	For any~$\epsilon \in (0,1)$,~$m \in \mathbb N$ and~$\lambda > 0$, there exists~$\delta_0 > 0$ such that the following holds. Let~$G$ be an~$(\epsilon,m)$-good graph with~$n$ vertices. Letting~$(\bm{\xi}_t)_{t \ge 0}$ be the contact process on~$G$ with parameter~$\lambda$ started from an arbitrary non-empty configuration, 
	\begin{itemize}
		\item[$\mathrm{(i)}$] for any~$i$ in the core, the probability that~$i$ is lit at time~$2$ is larger than~$\delta_0$;
		\item[$\mathrm{(ii)}$] for any~$i$ in the core and~$t > 0$, the probability that the process has not gone extinct by time~$t$, but~$i$ has not become lit until then, is smaller than~$(1-\delta_0)^{\lfloor t/2\rfloor}$.
	\end{itemize}
\end{lemma}
\begin{proof}
	It follows from the definition of an~$(\epsilon,m)$-good graph that all vertices in the core are neighbours of the vertex with index~$1$, and satellites have at most~$m$ vertices. This implies that the diameter of~$G$ is at most~$2m+4$. Hence, for a fixed vertex~$i$ in the core, starting from an arbitrary non-empty configuration, we can take~$\delta_1 > 0$ depending on~$m$ and~$\lambda$ so that with probability at least~$\delta_1$, the infection reaches~$i$ by time~$1$ and there is no recovery at~$i$ during~$[0,1]$. Thus~$i$ is infected at time~$1$, and the Markov property and Lemma~\ref{lem_all_stars}$(\mathrm{a})$ give conditional probability at least~$c_{\mathrm{star}}\cdot (\lambda \wedge 1)$ that~$i$ is lit at time~$2$. Hence,~$\mathrm{(i)}$ holds with~$\delta_0:=\delta_1 \cdot c_{\mathrm{star}}\cdot (\lambda \wedge 1)$.

The second statement then follows from the first, using a standard argument involving repeated trials.
\end{proof}

\begin{proposition}\label{prop_bushes_die_in_good}
For~$\epsilon \in (0,1/2)$ small enough,~$m \in \mathbb N$,~$\lambda_\mathcal{T} > 0$ and~$\lambda_{\mathcal{B}} > 0$, there exists~$C_0 > 0$ such that the following holds for~$n$ large enough. Let~$G$ be an~$(\epsilon,m)$-good graph with~$n$ vertices, and consider the grass-bushes-trees process on~$G$ with rates~$\lambda_{\mathcal T}$ and~$\lambda_{\mathcal B}$. Assume that for the tree configuration at time~$0$, all vertices in the core are semi-lit.
	Then,
\[
\mathbb P(\text{there are no bushes at time } C_0\log n) > 1 - \frac{1}{2n}.
\]
\end{proposition}
We postpone the proof to the next section.

\begin{proof}[Proof of Theorem~\ref{thm_main_heavy}]
	Take~$\epsilon > 0$ and~$m \in \mathbb N$ as required by Proposition~\ref{prop_graph_is_good}, and also assume that~$\epsilon$ is small enough as required by Proposition~\ref{prop_bushes_die_in_good}. We fix a deterministic graph~$G$ such that~$G$ is~$(\epsilon,m)$-good with~$n$ vertices. Letting~$(\mathcal T_t,\mathcal B_t)_{t \ge 0}$ be the grass-bushes-trees process on~$G$ starting from an arbitrary initial configuration, we will show that, if~$n$ is large enough,
	\begin{equation}\label{eq_complete_Markov}
		\mathbb P(\mathcal T_{2C_0 \log n} \neq \varnothing,\; \mathcal B_{2C_0\log n} \neq \varnothing) < \frac{1}{n}.
	\end{equation}
	The statement of the theorem then follows from this bound and Proposition~\ref{prop_graph_is_good}.

Abbreviating~$t_0:=C_0\log n$, we have
	\begin{align*}
		\mathbb P(\mathcal T_{2t_0} \neq \varnothing,\;\mathcal B_{2t_0} \neq \varnothing ) \le \mathbb P(\mathcal T_{t_0} \neq \varnothing,\;\mathcal B_{2t_0} \neq \varnothing ).
	\end{align*}
	Letting
	\[
		A_i:= \{i \text{ is semi-lit in $\mathcal T_{t_0}$} \}\quad \text{and} \quad A:=\bigcap_{i \in \text{core}}A_i,
	\]
	we bound 
	\begin{equation}\label{eq_complete_Markov0}
	\mathbb P(\mathcal T_{t_0} \neq \varnothing,\;\mathcal B_{2t_0} \neq \varnothing ) \le\mathbb P(\{\mathcal T_{t_0} \neq \varnothing\} \cap A^c) + \mathbb E[ \mathbb P ( \mathcal B_{2t_0} \neq \varnothing \mid (\mathcal T_{t_0}, \mathcal B_{t_0})  ) \cdot  \mathds{1}_A].
	\end{equation}
	By the Markov property and Proposition~\ref{prop_bushes_die_in_good}, 
	\begin{equation}
		\label{eq_complete_Markov1}
		\mathbb E[ \mathbb P ( \mathcal B_{2t_0} \neq \varnothing \mid (\mathcal T_{t_0}, \mathcal B_{t_0})  ) \cdot  \mathds{1}_A] \le \frac{1}{2n}.
	\end{equation}
	Next, for each~$i$ in the core, let~$\sigma_i := \inf\{t: i \text{ is lit in }\mathcal T_t \}$. We then bound
	\begin{align*}
		\mathbb P(\{\mathcal T_{t_0} \neq \varnothing\} \cap A_i^c) \le \mathbb P(\{ \mathcal T_{t_0/2} \neq \varnothing\} \cap \{\sigma_i > t_0/2\}) + \mathbb P(A_i^c \mid \sigma_i \le t_0/2).
	\end{align*}
	Lemma~\ref{lem_core_gbt}$\mathrm{(ii)}$ implies that
	\[
		\mathbb P(\{ \mathcal T_{t_0/2} \neq \varnothing\} \cap \{\sigma_i > t_0/2\}) \le (1-\delta_0)^{\lfloor t_0/4\rfloor} \le \exp\left\{-\delta_0 \cdot \left\lfloor \frac{C_0 \log n}{4}\right\rfloor \right\} < \frac{1}{n^2},
	\]
	if~$n$ is large enough.	Lemma~\ref{lem_all_stars}$\mathrm{(b)}$ implies that
	\[
		\mathbb P(A_i^c \mid \sigma_i \le t_0/2) \le \exp\{-c_{\mathrm{star}}(\lambda_{\mathcal T}^2 \wedge 1) \deg(i) \} \stackrel{\eqref{eq_bounds_degrees}}{\le} \exp\{-c_{\mathrm{star}}(\lambda_{\mathcal T}^2 \wedge 1)n^{\epsilon}/(\log n)^{10} \} < \frac{1}{n^2}
	\]
	for~$n$ large.
A union bound over~$i$ in the core then shows that
	\begin{equation}
		\label{eq_complete_Markov2}
		\mathbb P(\{\mathcal T_{t_0} \neq \varnothing\} \cap A^c)  \le n^{1-\epsilon}\cdot \frac{2}{n^2} < \frac{1}{2n}.
	\end{equation}
	Using~\eqref{eq_complete_Markov1} and~\eqref{eq_complete_Markov2} to bound~\eqref{eq_complete_Markov0}, we obtain~\eqref{eq_complete_Markov}.
\end{proof}

\subsection{The bonsai process}
In this section, we define an auxiliary process which will play a key role in the proof of Proposition~\ref{prop_bushes_die_in_good}. We take
\begin{equation}\label{eq_choice_of_eps} 
0< \epsilon <\min\left\{\frac18,\frac{2-\tau}{2(\tau-1)}\right\};\end{equation}
	the reason for this choice is hard to summarize now, as it has to do with several restrictions involving other constants that will appear, but it will become clear at the very end of this section. Fix~$\lambda_{\mathcal T} > 0$, and write
\begin{equation*}
	\Lambda_{\mathcal T}:= \lambda_{\mathcal T} \wedge 1.
\end{equation*}
Also fix~$m \in \mathbb N$.
	Let~$G=(V,E)$ be an~$(\epsilon,m)$-good graph on~$n$ vertices.

\begin{definition}
The \emph{bonsai process}~$(\bm{\beta}_t)_{t \ge 0}$ is the continuous-time Markov chain on~$\{0,1\}^{V}$ in which:
\begin{itemize}
\item if~$i$ is in the satellites,~$\bm{\beta}_t(i) \equiv 0$;
\item if~$i$ is in the core,~$(\bm{\beta}_t(i))_{t \ge 0}$ jumps from~$0$ to~$1$ with rate~$\Lambda_{\mathcal T}^2\deg(i)/(32e^3)$, and jumps from~$1$ to~$0$ with rate~$1$; this evolution is independent for distinct~$i$.
\end{itemize}
\end{definition}
As is usual, we associate a configuration~${\beta} \in \{0,1\}^{V}$ with the set~$\{i:{\beta}(i)=1\}$; we regard this as the set of vertices occupied by little bonsai trees.

The bonsai process is much easier to analyze than the tree process (that is, the contact process), as it is just a product chain. We will prove that the bonsai process is dominated by the tree process in any period where in the latter, each vertex~$i$ in the core has a fraction at least~$\Lambda_{\mathcal T}/(32e^3)$ of neighbours occupied by trees. 

The domination is materialized by a coupling of $(\boldsymbol{\beta}_t, \mathcal{T}_t)_t$ given as follows. 
For each $i$ in the core, the coupled chain performs the following jumps (all rates are evaluated at the current state $(\boldsymbol{\beta}, \mathcal{T})$):
\begin{center}
\renewcommand{\arraystretch}{1.4}
\begin{tabular}{ll}
Transition & Rate \\
\hline
$(\boldsymbol{\beta}, \mathcal{T}) \to (\boldsymbol{\beta} \setminus \{i\},\ \mathcal{T} \setminus \{i\})$ 
  & $1$ \\[.1cm]
$(\boldsymbol{\beta}, \mathcal{T}) \to (\boldsymbol{\beta} \cup \{i\},\ \mathcal{T} \cup \{i\})$ 
  &$ \min\left(\Lambda_{\mathcal T}^2\dfrac{\deg(i)}{32e^3},\ \lambda_{\mathcal T}|\{j \in \mathcal{T} : j \sim i\}|\right)$  \\[.1cm]
$(\boldsymbol{\beta}, \mathcal{T}) \to (\boldsymbol{\beta},\ \mathcal{T} \cup \{i\})$ 
  & $\Big(\lambda_{\mathcal{T}}|\{j \in \mathcal{T} : j \sim i\}| -\Lambda_{\mathcal{T}}^2 \dfrac{\deg(i)}{32e^3}\Big)^{\!+}$ \\[.1cm]
 $(\boldsymbol{\beta}, \mathcal{T}) \to (\boldsymbol{\beta} \cup \{i\},\ \mathcal{T})$
  & $\Big(\Lambda_{\mathcal{T}}^2\dfrac{\deg(i)}{32e^3} - \lambda_{\mathcal{T}}|\{j \in \mathcal{T} : j \sim i\}|\Big)^{\!+}$ 
\end{tabular}
\end{center}
In a satellite vertex~$i$, bonsais can never be present (in agreement with the definition of the bonsai dynamics), and trees evolve like a regular contact process; that is, transitions are:
\begin{center}
\renewcommand{\arraystretch}{1.4}
\begin{tabular}{ll}
Transition & Rate \\
\hline
   $(\bm{\beta},\mathcal{T}) \to (\bm{\beta},\ \mathcal{T}\setminus\{i\})$ 
  & $1$ \\
   $(\bm{\beta},\mathcal{T}) \to (\bm{\beta},\ \mathcal{T}\cup\{i\})$ 
  & $\lambda_{\mathcal T}|\{j \in \mathcal{T}: j \sim i\}|$ 
\end{tabular}
\end{center}
It is then elementary that the marginal processes have the correct rates of the bonsai process and the tree process, respectively. 

\begin{lemma}\label{lem_coupling_bonsai} The above coupling~$(\bm{\beta}_t,\mathcal T_t)_t$ satisfies the following.
Let~$i$ be a core vertex and assume that~$|\{j \in \mathcal T_t: j \sim i\}| \ge \Lambda_{\mathcal T}\frac{\deg(i)}{32e^3}$ for all~$t \in [0,T]$ for some~$T > 0$. Then,
\begin{itemize}
\item if~$(\bm{\beta}_t(i))_t$ has a jump at time~$s \in [0,T]$, then~$\mathcal T_s(i)\ge \bm{\beta}_s(i)$;
\item if~$\mathcal T_r(i) \ge \bm{\beta}_r(i)$ for some~$r \in [0,T]$, then~$\mathcal T_t(i) \ge \bm{\beta}_t(i)$ for all~$t \in [r,T]$.
\end{itemize}
\end{lemma}
We omit the proof, since it is elementary.

We now consider a \emph{grass-bushes-bonsais process}~$(\mathcal B_t,\bm{\beta}_t)_{t \ge 0}$ on~$G$ analogously to how the grass-bushes-trees process is defined (with the bonsais playing the roles of trees),  \emph{with one key difference}: bushes cannot die spontaneously at the core, and only disappear by being overwritten by bonsais.
More precisely, the dynamics has rates given as follows.
 For a core vertex~$i$, the transitions are
\begin{center}
\renewcommand{\arraystretch}{1.4}
\begin{tabular}{ll}
Transition &Rate  \\
\hline
   $(\bm{\beta},\mathcal{B}) \to (\bm{\beta} \setminus \{i\},\ \mathcal{B})$ 
  & $1$ \\[.1cm]
  $(\bm{\beta},\mathcal{B}) \to (\bm{\beta}\cup \{i\},\ \mathcal{B} \setminus \{i\})$ 
  & $\Lambda_{\mathcal T}^2 \frac{\deg(i)}{32e^3}$\\[.1cm]
  $(\bm{\beta},\mathcal{B}) \to (\bm{\beta},\ \mathcal{B} \cup \{i\})$ 
  & $\lambda_{\mathcal B}|\{j \in \mathcal B: j \sim i\}| \cdot \mathds{1}\{i \notin \bm{\beta} \}$ 
\end{tabular}
\end{center}
For a satellite vertex~$i$, the transitions are
\begin{center}
\renewcommand{\arraystretch}{1.4}
\begin{tabular}{ll}
Rate & Transition \\
\hline
   $(\bm{\beta},\mathcal{B}) \to (\bm{\beta},\ \mathcal{B}\setminus\{i\})$ 
  & $1$ \\[.1cm]
  $(\bm{\beta},\mathcal{B}) \to (\bm{\beta},\ \mathcal{B} \cup \{i\})$ 
  & $\lambda_{\mathcal B}|\{j \in \mathcal B: j \sim i\}| $ 
\end{tabular}
\end{center}

We now state a lemma that says that, although they are weaker than trees, the bonsais are already enough to eliminate all bushes in a small amount of time.
\begin{lemma}\label{lem_bonsais_win1}
There exist constants~$c,C > 0$ such that the following holds for~$n$ large enough.
For the grass-bushes-bonsais process on an~$(\epsilon,m)$-good graph on~$n$ vertices, started from an arbitrary initial configuration, the probability that there are no bushes at time~$t$ is larger than~$1-n^{C}e^{-ct}$.
\end{lemma}

We postpone the proof to Section~\ref{ss_bonsais_win}.

\begin{proof}[Proof of Proposition~\ref{prop_bushes_die_in_good}]
	Take~$t=C_0\log n$, with~$C_0$ to be chosen large enough.
	We take the coupling~$(\bm{\beta}_t,\mathcal{T}_t)_t$ described above of the bonsai and tree processes, with~$\bm{\beta}_0 = \varnothing$ and~$\mathcal T_0$ equal to the initial configuration given in the statement of the lemma. Independently of these, we take a graphical construction~$H$ of the contact process with parameter~$\lambda_{\mathcal B}$ on~$G$. We let~$H'$ be the modification of~$H$ in which we remove all recovery marks at core vertices. We use~$H$ to construct bushes (denoted~$(\mathcal B_t)_t$) competing against the trees in~$(\mathcal T_t)_t$, and we use~$H'$ to construct bushes (denoted~$(\mathcal B_t')_t$) competing against the bonsais in~$(\bm{\beta}_t)_t$.

	Letting
	\[
		T_{\mathrm{coup}} := \sup\{t: \bm{\beta}_s \le \mathcal T_s \text{ for all }s \le t\},
	\]
	we clearly have
	\[
		\mathcal B_t' \supseteq \mathcal B_t \quad \text{for all } t \le T_{\mathrm{coup}}.
	\]
	Hence, 
	\[
		\mathbb P(\mathcal B_t \neq \varnothing) \le \mathbb P(T_{\mathrm{coup}} \le t) + \mathbb P(\mathcal B_t' \neq \varnothing).
	\]
Lemma~\ref{lem_bonsais_win1} implies that
	\begin{equation}\label{eq_1st_rhs}
		\mathbb P(\mathcal B_t' \neq \varnothing) \le n^Ce^{-ct}
	\end{equation}
	for some~$c,C > 0$. Next, Lemma~\ref{lem_coupling_bonsai} gives
	\[
		\mathbb P(T_{\mathrm{coup}} \le t) \le \sum_{i\in \mathrm{core}}\mathbb P(\exists s \le t: \text{$i$ is not weakly lit in $\mathcal T_s$}).
	\]
	By~\eqref{eq_bounds_degrees}, the exponential time horizon in Lemma~\ref{lem_all_stars} exceeds~$t$ at every core vertex when~$n$ is large. Since all core vertices are initially semi-lit, we apply part~$(\mathrm{c})$ on~$[0,1]$ and then part~$(\mathrm{b})$ from time~$1$, when the stars are lit. The right-hand side above is therefore at most
	\begin{equation}\label{eq_2nd_rhs}
		2\sum_{i \in \mathrm{core}}\exp\left\{-c_{\mathrm{star}}\cdot (\Lambda_{\mathcal{T}})^2\cdot  \deg(i)\right\} \stackrel{\eqref{eq_bounds_degrees}}{\le} 2\sum_{i=1}^{\lfloor n^{1-\epsilon}\rfloor} \exp\left\{ -c_{\mathrm{star}}\cdot (\Lambda_{\mathcal{T}})^2\cdot \frac{n}{(\log n)^{10}i}\right\}.
	\end{equation}
	Taking~$C_0$ large enough that~$cC_0>C+1$, the right-hand sides of both~\eqref{eq_1st_rhs} and~\eqref{eq_2nd_rhs} are smaller than~$1/(4n)$ when~$n$ is large enough, completing the proof.
\end{proof}

\subsection{Bonsais destroy bushes}
\label{ss_bonsais_win}
In this section, we prove Lemma~\ref{lem_bonsais_win1}.
We work on an~$(\epsilon,m)$-good graph~$G$, and let
\[\mathscr C:=\{1,\ldots, \lfloor n^{1-\epsilon}\rfloor\}\]
denote its core.

In proving Lemma~\ref{lem_bonsais_win1}, by monotonicity, it is enough to prove that the configuration at time~$0$ is~$\mathcal B_0=V$ and~$\bm{\beta}_0=\varnothing$, so we assume this throughout the section.

It will be useful to have the evolution of the bushes given by a Poisson graphical construction that is killed upon touching bonsais. More precisely, we start by constructing the bonsai process~$(\bm{\beta}_t)_t$ with the rates in its definition, as a Markov chain on~$\{0,1\}^V$, started from~$\bm{\beta}_0=\varnothing$. 
We then take a graphical construction of the contact process on~$V$ with rate~$\lambda_{\mathcal B}$, \emph{except that it has no death marks in the core}, and set
\[
\mathcal B_t(i)=\mathds{1}\left\{
\vcenter{\hbox{\begin{varwidth}{0.7\linewidth}
there exists an infection path $\gamma:[0,t]\to V$ such that
$\gamma(0)\in \mathcal B_0$, $\gamma(t)=i$, and
$\gamma(s)\notin \bm{\beta}_s$ for all $s\in[0,t]$
\end{varwidth}}}
\right\}.
\]
Then, the pair~$(\bm{\beta}_t,\mathcal B_t)_t$ has the dynamics of the grass-bushes-bonsais process.

Let~$S$ be a satellite, that is, a connected component of the subgraph induced by~$\{i \in V: i > n^{1-\epsilon}\}$. For each~$t \ge 0$, we will define a positive random variable~$\sigma_S(t) \ge t$, which will depend on~$\mathcal B_t \cap S$ and the graphical construction on~$S \times [t,\infty)$. In case~$\mathcal B_t \cap S = \varnothing$, we set~$\sigma_S(t)=t$. If~$\mathcal B_t \cap S \neq \varnothing$, we set
\[
\sigma_S(t):=\sup\left\{ 
\vcenter{\hbox{\begin{varwidth}{0.7\linewidth}
$s \ge t$: there is an infection path from $(\mathcal B_t \cap S) \times \{t\}$ to $S \times \{s\}$ that stays inside $S$
\end{varwidth}}}
\right\}.
\]
This means that, using the graphical construction for the bushes, we run a contact process on~$S$ started at time~$t$ from~$\mathcal B_t \cap S$; we then define~$\sigma_S(t)\ge t$ as the extinction time of this process. 

%

It is worth it to observe that, in a time interval where~$S$ has bushes but receives no bush from the outside, the function~$t\mapsto \sigma_S(t)$ stays constant.

Our next step is to define a filtration~$(\mathscr F_t)_{t \ge 0}$. For each~$t \ge 0$, we let~$\mathscr F_t$ be the~$\sigma$-algebra generated by:~(1) the values~$(\mathcal B_s)_{0 \le s \le t}$,~(2) the graphical construction of the bush process on~$[0,t]$, and~(3) the values of~$\sigma_S(s)$ for all satellites~$S$ and all~$s \in [0,t]$.
\begin{remark}
	Putting aside item~(3) of the above list, the reader may be intrigued by the inclusion of~(1) and~(2) simultaneously, as they may seem to be redundant together. That is, however, not the case. An observer of the bush process who is kept ignorant of the graphical construction (and of the bonsais) would not know the parent of a bush that has just appeared, and would not know the times when bushes attempted to give birth but were unsuccessful due to the presence of bonsais. On the other hand, an observer of the graphical construction who is kept ignorant of the bush process (and of the bonsais) cannot fully recover the bush process, because of ignorance of the bonsais.
\end{remark}

We now let~$\eta, \kappa, \alpha > 0$ be small constants whose values will be chosen later. We define
\begin{align}
&X^{(i)}_t:=\left( \frac{n}{i}\right)^{\kappa} \cdot  \mathds{1}\{ i \in \mathcal B_t \},\qquad i \in \mathscr C,\; t \ge 0,\\[.2cm]
	&X^{(S)}_t:= n^\eta  e^{\alpha(\sigma_S(t)-t) } \cdot  \mathds{1}\{\mathcal B_t \cap S \neq \varnothing\},\qquad S \text{ satellite},\; t \ge 0, \label{eq_with_ind}
\end{align}
and
\[
	X_t:=\sum_{i \in \mathscr C} X^{(i)}_t + \sum_{S \text{ satellite}} X^{(S)}_t,\qquad M_t:=e^{\frac{\alpha}{2} t}X_t,\qquad t \ge 0.
\]

\begin{lemma}\label{lem_super}
For~$\eta,\kappa,\alpha$ suitably chosen, the process~$(M_t)_{t \ge 0}$ is a supermartingale with respect to~$(\mathscr{F}_t)_{t \ge 0}$.
\end{lemma}

Before we give the proof, let us see how this allows us to prove Lemma~\ref{lem_bonsais_win1}.

\begin{proof}[Proof of Lemma~\ref{lem_bonsais_win1}]
	For~$t \ge 0$, we bound
\[
	\mathbb P(\mathcal B_t \neq \varnothing) \le \mathbb E[X_t] = e^{-\frac{\alpha}{2} t} \cdot \mathbb E[M_t] \le e^{-\frac{\alpha}{2} t}\cdot  \mathbb E[M_0],
\]
where the first inequality follows from~$X_t \ge \mathds{1}\{X_t > 0\}$.
We have
\[
\mathbb E[M_0] \le \sum_{i \in \mathscr C} \left(\frac{n}{i}\right)^\kappa+ n^{\eta} \cdot \left( \sup_{X} \mathbb{E}[e^{\alpha X}]\right) \cdot |\{\text{satellites}\}|,
\]
where the supremum is taken over all random variables~$X$ whose distribution is that of the extinction time of the contact process with rate~$\lambda_{\mathcal B}$ on graphs with at most~$m$ vertices (starting from full occupancy). This shows that~$\mathbb E[M_0]$ grows at most polynomially in~$n$, concluding the proof.
\end{proof}

We need some preliminary results to prove Lemma~\ref{lem_super}. We separately derive bounds for the core and bounds for the satellites, and then put them all together.

\subsubsection{Bounds for the core}
We start with the following lemma, whose proof we omit as it is elementary, using the definitions of the bonsai and bush dynamics and of the filtration~$(\mathscr F_t)_t$.
\begin{lemma}\label{lem_derivative_core}
	For all~$i \in \mathscr C$, on the event~$\{i \in \mathcal B_t\}$, we have
	\[
		\left.\frac{\mathrm{d}}{\mathrm{d}h} \mathbb E\big[X^{(i)}_{t+h} \mid \mathscr F_t\big]\right|_{h=0+} = - \frac{\Lambda_{\mathcal{T}}^2 \deg(i)}{32e^3}\cdot  \left( \frac{n}{i}\right)^\kappa,
	\]
	and on the event~$\{i \notin \mathcal B_t\}$,
	\[
		\left.\frac{\mathrm{d}}{\mathrm{d}h} \mathbb E\big[X^{(i)}_{t+h} \mid \mathscr F_t\big]\right|_{h=0+} = \lambda_{\mathcal B}\cdot  |\{j: j \in \mathcal B_t,\\ j \sim i \}|\cdot \mathbb P(i \notin \bm{\beta}_t \mid \mathscr{F}_t)\cdot \left( \frac{n}{i} \right)^{\kappa}.
	\]
\end{lemma}

We now estimate the probability~$\mathbb P(i \notin \bm{\beta}_t \mid \mathscr{F}_t)$. This is where we make use of the choice of not allowing the bushes to die spontaneously in the core.
\begin{lemma}\label{lem_thinning}
	For all~$i \in \mathscr C$, on the event~$\{i \notin \mathcal B_t\}$, we have
\[
	\mathbb P(i \notin \bm{\beta}_t \mid \mathscr{F}_t) \le \frac{64e^3(\log n)^{10}}{\Lambda_{\mathcal T}^2}\cdot  \frac{i}{n}. 
\]
\end{lemma}

\begin{proof}
	Fix~$t > 0$.
	Fix~$i \in \mathscr C$ with~$i \notin \mathcal B_t$. Let~$\mathsf T_i:=\sup\{s \le t: i \in \mathcal B_s\}$. We have~$\mathsf T_i > 0$, because we assume that all vertices have bushes at time~$0$, and~$\mathsf T_i < t$, because we assume~$i \notin \mathcal B_t$. More importantly, since bushes cannot spontaneously recover in the core, we know that the bush that was at~$i$ immediately before~$\mathsf T_i$ must have been displaced by a bonsai, so~$i \in \bm{\beta}_{\mathsf T_i}$.

	Define
	\[
		\mathcal R:= \left\{\vcenter{\hbox{\begin{varwidth}{0.5\linewidth}
			$s \in (\mathsf T_i,t]: \exists j \in V: j \in \mathcal B_s,\; j \sim i,$ there is a birth arrow from $(j,s)$ to $(i,s)$
\end{varwidth}}} \right\},
	\]
	that is,~$\mathcal R$ is the set of times in~$(\mathsf T_i,t]$ when a bush attempts to send a bush to~$i$. We know all these attempts are unsuccessful, because the definition of~$\mathsf T_i$ implies that~$i$ is not occupied by a  bush in the time interval~$(\mathsf T_i,t]$. Hence, we can infer that~$i$ was occupied by bonsais at all times of~$\mathcal R$. Note that~$\mathcal R$ is measurable with respect to~$\mathscr F_t$.

	We now separately treat the cases~$\mathcal R = \varnothing$ and~$\mathcal R \neq \varnothing$, starting with the former. In that case,~$\mathscr F_t$ carries no information about the behaviour of the bonsais at~$i$ between times~$\mathsf T_i$ and~$t$, apart from the fact that~$i \in \bm{\beta}_{\mathsf T_i}$. Then, conditionally on~$\mathscr F_t$, the bonsai process on~$[\mathsf T_i,t]$ evolves like a two-state Markov chain on~$\{0,1\}$ (with~$0=\text{empty}$ and~$1=\text{bonsai}$) that jumps from~$1$ to~$0$ with rate~$1$, and from~$0$ to~$1$ with rate~$\Lambda_{\mathcal T}^2 \frac{\deg(i)}{32e^3}$. Since this chain is in state~$1$ at time~$\mathsf T_i$, its distribution at time~$t$ stochastically dominates its equilibrium, so it is in state~0 at time~$t$ with (conditional) probability at most
	\[\frac{1}{\Lambda_{\mathcal T}^2 \deg(i)/(32e^3)+1} \le \frac{64e^3}{\Lambda_{\mathcal T}^2 \deg(i)} \le \frac{64e^3(\log n)^{10}}{\Lambda_{\mathcal T}^2}\cdot  \frac{i}{n},\]
	by~\eqref{eq_bounds_degrees}.

	The case~$\mathcal R \neq \varnothing$ is treated in the same way, except that the time interval~$[\mathsf T_i, t]$ is replaced by~$[\max \mathcal R,t]$. This concludes the proof.
\end{proof}

\begin{corollary}\label{cor_for_gronwall}
Defining
	\begin{equation*}
		\mathsf{d}(\mathcal B_t,i):=\frac{\Lambda_{\mathcal T}^2}{32e^3(\log n)^{10}} \cdot \left( \frac{n}{i}\right)^{1+\kappa}\text{ for } i \in \mathscr C \cap \mathcal{B}_t
	\end{equation*}
	and
	\begin{equation*}
		\mathsf{b}(\mathcal B_t, i):= \frac{64e^3\lambda_{\mathcal B}(\log n)^{10}}{\Lambda_{\mathcal T}^2} \cdot \left( \frac{i}{n}\right)^{1-\kappa} \text{ for } i \in \mathscr C \setminus \mathcal{B}_t,
	\end{equation*}
	we have
	\begin{align*}
		\sum_{i \in \mathscr C} \left.\frac{\mathrm{d}}{\mathrm{d}h} \mathbb E\big[X^{(i)}_{t+h}\mid \mathscr{F}_t \big]\right|_{h=0+} \le  -\sum_{i \in \mathscr C \cap \mathcal B_t}\mathsf d(\mathcal B_t,i) + \sum_{\substack{(i,j): i \in \mathcal B_t,\\ j \in \mathscr C \setminus \mathcal B_t,\\ j \sim i}} \mathsf b(\mathcal B_t,j).
	\end{align*}
\end{corollary}
\begin{proof}
	We apply Lemma~\ref{lem_derivative_core} to each term of the sum. For~$i \in \mathcal B_t$, we bound~$\deg(i)$ using~\eqref{eq_bounds_degrees}. For~$i \notin \mathcal B_t$, we bound~$\mathbb P(i \notin \bm{\beta}_t \mid \mathscr F_t)$ using Lemma~\ref{lem_thinning}. We then rearrange the sum to obtain the desired expression.
\end{proof}

\subsubsection{Bounds for the satellites}
Let~$S$ be a satellite of~$G$. Noting that the trajectory~$t \mapsto X^{(S)}_t$ is piecewise smooth, we write
\begin{equation}\label{eq_def_X_sat}
	X^{(S)}_t - X^{(S)}_0 = \sum_{0 < s \le t} (X^{(S)}_s - X^{(S)}_{s-}) - \alpha n^\eta \int_0^t e^{\alpha (\sigma_S(s)-s)} \cdot  \mathds{1}\{ \mathcal B_s \cap S \neq \varnothing \}\;\mathrm{d}s.
\end{equation}
This trajectory has positive and negative jumps. The positive jumps occur at times~$t$ when a bush from the core gives birth to a new bush inside~$S$, making~$\sigma_S(t)$ increase. The negative jumps occur at times~$t$ when the bush process becomes empty inside~$S$, making the indicator function in~\eqref{eq_with_ind} switch from~$1$ to~$0$. We will give close attention to the positive jumps, so we define
\begin{equation*}
	Y^{(S)}_t := \sum_{0 < s \le t} (X^{(S)}_s - X^{(S)}_{s-})^+.
\end{equation*}

\begin{lemma}
	Let~$S$ be a satellite of~$G$. For each vertex~$j$ of~$S$, define
	\[\mathsf E_j := \mathbb E\left[ e^{\alpha \mathsf T_j} \right], \]
	where~$\mathsf T_j$ is the extinction time of a contact process on~$S$ with rate~$\lambda_{\mathcal B}$, started from only~$j$ infected. Then,
	\[\left.\frac{\mathrm{d}}{\mathrm{d}h} \mathbb E\big[Y^{(S)}_{t+h} \mid \mathscr F_t\big]\right|_{h=0+} \le \lambda_{\mathcal B} n^\eta \sum_{\substack{(i,j):i \in \mathscr C \cap \mathcal B_t,\\ j \in S \setminus \mathcal B_t,\\ i \sim j}}   \mathsf E_j.\]
\end{lemma}
\begin{proof}
	The process~$(Y^{(S)}_t)_t$ has a jump whenever a bush from the core sends a bush to an empty site in~$S$. For each pair of neighboring vertices~$(i,j)$ with~$i \in \mathscr C \cap \mathcal B_t$ and~$j \in S \setminus \mathcal B_t$, this happens with rate~$\lambda_{\mathcal B}$.

	For~$A \subseteq S$, let~$\Phi(A,t)$ be the smallest~$t' \ge t$ such that there is no infection path inside~$S$ from~$A \times \{t\}$ to~$S \times \{t'\}$. By definition,~$\sigma_S(t)=\Phi(\mathcal B_t,t)$. Let us also put~$\Phi(\varnothing,t)=t$.

	Let us first consider the case where~$\mathcal B_t \cap S = \varnothing$. 
	If a bush is born at~$j \in S$ immediately after~$t$, then the term corresponding to~$S$ in the definition of~$X_t$ jumps from~$0$ to~$n^\eta\exp\{\alpha (\Phi(\{j\},t)-t)\}$.  The expected size of this jump conditionally on~$\mathscr F_t$ is
	\[
		n^\eta\mathbb E[\exp\{\alpha (\Phi(\{j\},t)-t)\}\mid \mathscr F_t] =n^\eta \mathsf E_j
	\]
	(here and in the next case, it is useful to recall that, in~$\mathscr F_t$, the only information concerning the graphical construction in~$S \times [t,\infty)$ is the value of~$\sigma_S(t) = \Phi(\mathcal B_t,t)$).

	We now turn to the case where~$\mathcal B_t \cap S \neq \varnothing$.
	If a bush is born at~$j \in S \setminus \mathcal B_t$ at a time immediately after~$t$, the expected jump in the term corresponding to~$S$ in~$X_t$ given~$\mathscr F_t$ is
	\begin{align*}
		&n^\eta \mathbb E\left[\left.\exp\{\alpha(\Phi(\mathcal B_t \cup \{j\},t)-t) \} - \exp\{\alpha(\Phi(\mathcal B_t ,t)-t) \}\; \right| \; \mathscr F_t\right].
	\end{align*}
	By Lemma~\ref{lem_prelim_moment}$\mathrm{(b)}$, this is bounded by~$n^\eta\mathsf E_j$.
\end{proof}

Define
\[
	\mathsf E := \sup_{\mathsf T} \mathbb E\big[ e^{\alpha \mathsf T}\big],
\]
where the supremum is taken over all random variables~$\mathsf T$, with distribution ranging over the distribution of the extinction time of the contact process with rate~$\lambda_{\mathcal B}$ on all graphs with up to~$m$ vertices, starting from a single vertex infected. Combining the above lemma with the bounds in~\eqref{eq_bounds_degrees}, we write:
\begin{equation}\label{eq_Y_derivative}
	\left.\frac{\mathrm{d}}{\mathrm{d}h} \mathbb E\big[Y^{(S)}_{t+h} \mid \mathscr F_t\big]\right|_{h=0+}  \le \lambda_{\mathcal B} n^\eta \mathsf E \cdot  |\{(i,j): i \in \mathscr C \cap \mathcal B_t,\; j \in S \setminus \mathcal B_t,\; i \sim j\}|.
\end{equation}

\subsubsection{Verification of supermartingale property}
We now let
\begin{equation}\label{eq_def_X_sat2}
	\widetilde{X}^{(S)}_t := Y^{(S)}_t - \alpha n^\eta \int_0^t e^{\alpha (\sigma_S(s)-s)} \cdot  \mathds{1}\{ \mathcal B_s \cap S \neq \varnothing \}\;\mathrm{d}s
\end{equation}
and
\begin{equation*}
	\widetilde{X}_t :=\sum_{i \in \mathscr C} X^{(i)}_t + \sum_{S \text{ satellite}} \widetilde{X}^{(S)}_t.
\end{equation*}
By comparing~\eqref{eq_def_X_sat} and~\eqref{eq_def_X_sat2}, for all~$s \le t$ we have
\begin{equation}\label{eq_explicit_Y}
	X_t - X_s \le \widetilde{X}_t - \widetilde{X}_s.
\end{equation}

\begin{lemma}\label{lem_new_gronwall}
	For~$\eta,\kappa$ suitably chosen and~$n$ large enough, we have that for every~$t \ge 0$,
	\begin{equation*}
		\left. \frac{\mathrm{d}}{\mathrm{d}h}\mathbb E\big[ \widetilde{X}_{t+h}\mid \mathscr F_t\big]\right|_{h=0+} \le - \frac{\alpha}{2} X_t.
	\end{equation*}
\end{lemma}

Before we prove this, let us show how it allows us to prove that~$(M_t)_t$ is a supermartingale.

\begin{proof}[Proof of Lemma~\ref{lem_super}]
	Using~\eqref{eq_explicit_Y}, for~$s < t$,
	\begin{align*}
		\mathbb E[X_t - X_s \mid \mathscr F_s] &\le \mathbb E\big[\widetilde{X}_t - \widetilde{X}_s \mid \mathscr F_s \big]= \int_s^t \frac{\mathrm{d}}{\mathrm{d}u} \mathbb E\big[\widetilde{X}_u \mid \mathscr F_s \big] \;\mathrm{d}u. 
	\end{align*}
	Using the tower property of conditional expectation, and exchanging the derivative with the expectation, 
	\[
		\frac{\mathrm{d}}{\mathrm{d}u} \mathbb E\big[\widetilde{X}_u \mid \mathscr F_s \big] = \mathbb E\left[\left.\left. \frac{\mathrm{d}}{\mathrm{d}h} \mathbb E\big[ \widetilde{X}_{u+h} \mid \mathscr F_u  \big] \right|_{h=0+}\; \right| \; \mathscr F_s \right].
	\]
	By Lemma~\ref{lem_new_gronwall}, the right-hand side is smaller than~$-\frac{\alpha}{2}\mathbb E[X_u \mid \mathscr F_s]$. Then, using Gr\"onwall's inequality on the function~$[s,t] \ni u \mapsto \mathbb E[X_u \mid \mathscr{F}_s]$, we obtain
	\[
		\mathbb E[X_t \mid \mathscr F_s] \le e^{-\frac{\alpha}{2}(t-s)} X_s,
	\]
	proving that~$(M_t)_t = (e^{\frac{\alpha}{2}t} X_t)_t$ is a supermartingale.
\end{proof}

It remains to prove Lemma~\ref{lem_new_gronwall}, which is where we need to get our hands dirty (from trimming bushes!), and choose the constants~$\eta$ and~$\kappa$ in the process.

\begin{proof}[Proof of Lemma~\ref{lem_new_gronwall}]
	Take~$\alpha > 0$ small enough that~$\mathsf E < \infty$, which is possible by Lemma~\ref{lem_prelim_moment}$\mathrm{(a)}$. Using Corollary~\ref{cor_for_gronwall}, the definition~\eqref{eq_def_X_sat2}, and the bound~\eqref{eq_Y_derivative} on the conditional derivative of~$Y^{(S)}$, we bound
	\begin{equation}\label{eq_ineq_with_ds}
		\left. \frac{\mathrm{d}}{\mathrm{d}h}\mathbb E\big[ \widetilde{X}_{t+h}\mid \mathscr F_t\big]\right|_{h=0+} \le \sum_{\substack{i \in \mathscr C \cap \mathcal B_t}} \mathcal{S}_1(\mathcal B_t, i) + \sum_{S: \mathcal B_t \cap S \neq \varnothing} \mathcal S_2(\mathcal B_t, S),
	\end{equation}
	where for~$i \in \mathscr C \cap \mathcal B_t$, we define
	\[
		\mathcal S_1(\mathcal B_t, i):=-\mathsf{d}(\mathcal B_t, i) + \sum_{\substack{j \in \mathscr C \setminus \mathcal B_t,\\ j \sim i}} \mathsf{b}(\mathcal B_t, j) + \lambda_{\mathcal B}n^{\eta} \mathsf E \cdot |\{j \in V \setminus \mathscr C: j \notin \mathcal B_t,\; j \sim i\}|, 
	\]
	and for a satellite~$S$ with~$S \cap \mathcal B_t \neq \varnothing$, we define
	\[
		\mathcal S_2(\mathcal B_t, S):= - \alpha n^\eta e^{\alpha(\sigma_S(t)-t)} +  \sum_{\substack{i \in S \cap \mathcal B_t }}\; \sum_{\substack{j \in \mathscr C \setminus \mathcal B_t,\\ j \sim i}} \mathsf{b}(\mathcal B_t, j).
	\]
	
The desired result will follow from proving that, for suitable choices of~$\eta,\kappa$, we have
	\begin{equation}\label{eq_S1_end_bound}
		\mathcal S_1(\mathcal B_t,i) \le - \frac{\alpha}{2} \left( \frac{n}{i}\right)^\kappa \quad\text{for all } i \in \mathscr C \cap  \mathcal B_t
	\end{equation}
	and
	\begin{equation}\label{eq_S2_end_bound}
		\mathcal S_2(\mathcal B_t,S) \le -\frac{\alpha}{2} n^{\eta} e^{\alpha(\sigma_S(t)-t) }  \quad \text{for all satellites } S \text{ with } S \cap \mathcal B_t \neq \varnothing.
	\end{equation}
	
	We will collect conditions on the small positive constants~$\eta$ and~$\kappa$ as we progress, and in the end we will make sure they can all be mutually satisfied.

	Let us prove~\eqref{eq_S1_end_bound}. For~$i \in \mathscr C \cap \mathcal B_t$,
	\begin{align*}
		\sum_{\substack{j \in \mathscr C \setminus \mathcal B_t,\\ j \sim i}} \mathsf b(\mathcal B_t,j) \le \frac{64e^3\lambda_{\mathcal B}(\log n)^{10}}{\Lambda_{\mathcal T}^2 } \sum_{\substack{j \in \mathscr C,\\ j \sim i}} \left( \frac{j}{n}\right)^{1-\kappa}.
	\end{align*}
	
	By bounding~$j/n \le 1$, the sum on the right-hand side is at most~$\deg(i)$, which we bound by~$(\log n)^{\mathsf k} \cdot n/i$ (using~\eqref{eq_bounds_degrees}). Next, using~\eqref{eq_bounds_degrees} for~$i \le n^\epsilon$ and~\eqref{eq_neighbours_in_sat} for~$i \ge n^\epsilon$, we bound
	\[
|\{j \in V \setminus \mathscr C: j \notin \mathcal B_t,\; j \sim i\}| \le \log n+(\log n)^{\mathsf k+\mathsf k'} \cdot \left(1 \wedge \left(\frac{n^\epsilon}{i} \right)^{\frac{2-\tau}{\tau-1}}\right) \cdot \frac{n}{i}.
	\]
	
	This shows that~$\mathcal S_1(\mathcal B_t,i)$ is at most
	\begin{align}
		\nonumber&- \frac{\Lambda_{\mathcal T}^2}{32e^3(\log n)^{10}} \cdot \left( \frac{n}{i}\right)^{1+\kappa} + \frac{64e^3\lambda_{\mathcal B}(\log n)^{10+\mathsf k}}{\Lambda_{\mathcal T}^2 } \cdot \frac{n}{i} + \lambda_{\mathcal B} n^\eta \mathsf E \log n\\[.2cm]
		\nonumber&\qquad + \lambda_\mathcal{B} (\log n)^{\mathsf k+\mathsf k'} \mathsf E \cdot \left(1 \wedge \left(\frac{n^\epsilon}{i} \right)^{\frac{2-\tau}{\tau-1}}\right)\cdot \frac{n^{1+\eta}}{i}\\[.2cm]
		&\le - \frac{\Lambda_{\mathcal T}^2}{64e^3(\log n)^{10}} \cdot \left( \frac{n}{i}\right)^{1+\kappa} + \lambda_\mathcal{B} (\log n)^{\mathsf k+\mathsf k'} \mathsf E \cdot \left(1 \wedge \left(\frac{n^\epsilon}{i} \right)^{\frac{2-\tau}{\tau-1}}\right)\cdot \frac{n^{1+\eta}}{i},\label{eq_expression_wedge}
	\end{align}
	the inequality being true for~$n$ large enough, provided that~$\eta < \epsilon(1+\kappa)$.
	
	We now consider two cases.
	\begin{itemize}
		\item In case~$i \le n^\epsilon$, by bounding~$\frac{n}{i} \ge n^{1-\epsilon}$,~$\frac{n^\epsilon}{i} \ge 1$ and~$\frac{n^{1+\eta}}{i} \le n^{1+\eta}$, the expression in~\eqref{eq_expression_wedge} is
	\[
		\le - \frac{\Lambda_{\mathcal T}^2}{64e^3(\log n)^{10}} \cdot n^{(1-\epsilon)(1+\kappa)} + \lambda_{\mathcal B} (\log n)^{\mathsf k+\mathsf k'} \mathsf E \cdot {n^{1+\eta}}.
	\]
	Now, assuming that
\begin{equation}
	\label{eq_restrictions_constants1}
	(1-\epsilon)(1+\kappa) > 1 + \eta,
\end{equation}
			and that~$n$ is large enough, the absolute value of the term with~$n^{(1-\epsilon)(1+\kappa)}$ is much larger than that of the term with~$n^{1+\eta}$. Also using~$(1-\epsilon)(1+\kappa) > \kappa$ (since~$\epsilon < 1/8$ and~$\kappa < 1$), we obtain
			\[
				\mathcal S_1(\mathcal B_t, i) \le -\frac{\alpha}{2} n^\kappa \le -\frac{\alpha}{2}\left(\frac{n}{i}\right)^\kappa.
			\]
		\item If~$n^\epsilon \le i \le n^{1-\epsilon}$, then the last expression in~\eqref{eq_expression_wedge} is equal to
			\begin{equation}\label{eq_other_S1}
				- \frac{\Lambda_{\mathcal T}^2}{64e^3(\log n)^{10}} \cdot \left( \frac{n}{i}\right)^{1+\kappa} + \lambda_\mathcal{B} (\log n)^{\mathsf k+\mathsf k'} \mathsf E \cdot \frac{n^{1+\eta+\frac{2-\tau}{\tau - 1}\epsilon}}{i^{1+\frac{2-\tau}{\tau-1}}}
			\end{equation}
			Writing~$a:=\frac{2-\tau}{\tau-1} > 0$, we now want~$(\frac{n}{i})^{1+\kappa} \gg \frac{n^{1+\eta+a\epsilon}}{i^{1+a}}$, that is,~$ i^{a-\kappa} \gg n^{{\eta+a\epsilon - \kappa}}$. Assuming that
			\begin{equation}\label{eq_restrictions_constants2}
				\kappa < a= \frac{2-\tau}{\tau-1} \qquad \text{and}\qquad \eta < \kappa(1-\epsilon),
			\end{equation}
and using~$i \ge n^\epsilon$, we have
			\[ i^{a-\kappa} \ge n^{\epsilon(a-\kappa)}\gg n^{\eta +a\epsilon-\kappa}.\]
 Then, the negative term in~\eqref{eq_other_S1} has much larger magnitude than the positive one, so, also using~$n/i \ge n^\epsilon$,
			\[\mathcal S_1(\mathcal B_t, i) \le - \frac{\Lambda_{\mathcal T}^2}{128e^3(\log n)^{10}} \cdot \left( \frac{n}{i}\right)^{1+\kappa}  \le -\frac{\alpha}{2} \left(\frac{n}{i}\right)^\kappa.
			\]
\end{itemize}

This completes the proof of~\eqref{eq_S1_end_bound}.

We now turn to~\eqref{eq_S2_end_bound}.
Fix a satellite~$S$ and~$i \in S$. We bound
\begin{align*}
	\sum_{\substack{j \in\mathscr C \setminus \mathcal B_t,\\ j \sim i}} \mathsf b(\mathcal B_t,j) \le \frac{64e^3\lambda_{\mathcal B}(\log n)^{10}}{\Lambda_{\mathcal T}^2} \sum_{\substack{j \in\mathscr C,\\ j \sim i}} \left( \frac{j}{n}\right)^{1-\kappa}.
\end{align*}

We bound~$\deg(i)\le (\log n)^{\mathsf k} n^\epsilon$ by~\eqref{eq_bounds_degrees}, and~$\frac{j}{n} \le n^{-\epsilon}$ for~$j \in \mathscr C$. The sum on the right-hand side above is then at most~$(\log n)^{\mathsf k} n^{\epsilon \kappa}$. Recalling that there are at most~$m$ vertices in any satellite, we obtain
\[\sum_{i \in S \cap \mathcal B_t}\sum_{\substack{j \in\mathscr C \setminus \mathcal B_t,\\ j \sim i}} \mathsf b(\mathcal B_t,j) \le \frac{64e^3\lambda_{\mathcal B} m (\log n)^{10+\mathsf k}}{\Lambda_{\mathcal T}^2} n^{\epsilon \kappa}.\]

Assuming that
\begin{equation}\label{eq_restrictions_constants3}
\eta > \epsilon \kappa,
\end{equation}
we obtain, using~$\sigma_S(t) \ge t$,
\[
 \sum_{i \in S \cap \mathcal B_t}\sum_{\substack{j \in \mathscr C \setminus \mathcal B_t,\\ j \sim i}} \mathsf{b}(\mathcal B_t, j)\ll \alpha n^\eta e^{\alpha(\sigma_S(t)-t)},
\]
so
\[
	\mathcal S_2(\mathcal B_t, S) \le -\frac{\alpha}{2}n^\eta e^{\alpha(\sigma_S(t)-t)}.
\]

Together with~$0 < \kappa < 1$ and~$\eta < \epsilon(1+\kappa)$, the constraints imposed by~\eqref{eq_restrictions_constants1},~\eqref{eq_restrictions_constants2}, and~\eqref{eq_restrictions_constants3} are, in summary:
\[0 < \kappa < \min\left\{1,\frac{2-\tau}{\tau-1}\right\},\quad \epsilon \kappa < \eta < \min\{\epsilon(1+\kappa), \kappa-\epsilon-\kappa \epsilon\}.\]

Recalling~\eqref{eq_choice_of_eps}, we can then take~$\kappa=2\epsilon$ and~$\eta \in (2\epsilon^2,\epsilon-2\epsilon^2)$ to satisfy both.
\end{proof}

\section{No coexistence: the case~$\lambda_{\mathcal B} \le \lambda_{\mathcal T}$}\label{s_second_no_coex}
In this section, we prove Theorem~\ref{thm_no_coex_lambdas} on displacement in polynomial time when $\lambda_{\mathcal B} \le \lambda_{\mathcal T}$. Given that Theorem~\ref{thm_main_heavy} already covers the case~$\tau \in (1,2)$ (even without the assumption that~$\lambda_{\mathcal B} \le \lambda_{\mathcal T}$), we assume throughout this section that~$\tau \ge 2$.

\subsection{Graph conditions}
As in Section~\ref{s_no_coex}, we identify good graph conditions that happen with high probability, and that help in proving that there is no coexistence of bushes and trees.

\begin{lemma}\label{lem_diameter}
There exists~$C > 0$ such that with high probability as~$n \to \infty$, the diameter of~$G_n$ is smaller than~$C\log n$.
\end{lemma}

\begin{lemma}\label{lem_max_degree}
	There exists~$\gamma > 0$ such that with high probability as~$n \to \infty$,~$G_n$ has a vertex of degree larger than $n^\gamma$.
\end{lemma}

The statements of the two lemmas (and in fact much more precise results) are well known in the literature of the configuration model (with no erasing) with~$\tau > 2$.  We have two issues: the facts that we take the erased version of the model (which has no effect on the diameter, but does affect the maximum degree) and that we also need the case~$\tau = 2$. In Section~\ref{s_proof_graph_2}, we give sketches of proofs of the lemmas, referring to the literature when possible.

\subsection{Proof of extinction of bushes}
\begin{proof}[Proof of Theorem~\ref{thm_no_coex_lambdas}]
Since the trees evolve autonomously, the grass-bushes-trees process on an arbitrary graph with an arbitrary initial configuration satisfies, for any~$t$,
\[
	\lambda_\mathcal{B} \mapsto \mathbb P(\mathcal T_t \neq \varnothing, \mathcal B_t \neq \varnothing) \text{ is non-decreasing.}
\]
For this reason, it is sufficient to prove the theorem under the assumption that~$\lambda_{\mathcal B}=\lambda_{\mathcal T}=:\lambda$. This has the advantage that we can construct the process with a single graphical construction with parameter~$\lambda$, governing the behaviour of both bushes and trees.

By Lemma~\ref{lem_connectedness}, Lemma~\ref{lem_diameter}, and Lemma~\ref{lem_max_degree}, there exist~$\gamma > 0$ and~$C > 0$ such that~$\mathbb P(G_n \in \mathcal G_n) \xrightarrow{n \to \infty} 1$, where~$\mathcal G_n$ is the set of graphs with~$n$ vertices that are connected, have diameter smaller than~$C \log n$, and have a vertex of degree higher than~$n^\gamma$.

Let~$G$ be a graph in~$\mathcal G_n$, and let~$v^*$ be a vertex of~$G$ with~$\deg(v^*) \ge n^\gamma$. Let~$T:=n^k$, where~$k$ is a large constant to be chosen later.

We assume a graphical construction for the contact process with parameter~$\lambda$ on~$G$ is given. We will denote by~$\bm{\xi}_t^A$ the contact process on~$G$ started from~$\mathds{1}_A$ (and write~$\bm{\xi}^x_t$ instead of~$\bm{\xi}^{\{x\}}_t$). Later we will switch to the grass-bushes-trees process (denoted~$(\mathcal T_t,\mathcal B_t)_t$, with some specified initial configuration) under this same graphical construction. Although the contact process and the tree process are actually the same, it will be useful to keep both notations~$\bm{\xi}$ and~$\mathcal T$, and only use the latter when we are considering trees jointly with bushes. 

For each vertex~$x$ of~$G$, define the event
\[
	E_x := \left\{ \text{either }\bm{\xi}^{x}_T = \varnothing \text{ or } \int_{T/3}^{2T/3} \bm{\xi}^{x}_t(v^*)\;\mathrm{d}t > \frac{T}{6} \right\}.
\]
For each vertex~$x$, define the dual process
\[
	\hat{\bm{\xi}}^{(x,T)}_t(y):= \mathds{1}\{(y,T-t) \rightsquigarrow (x,T)\},\quad y \in G,\; t \in [0,T]
\]
and define the event
\[
	E_x':= \left\{ \text{either } \hat{\bm{\xi}}^{(x,T)}_T = \varnothing \text{ or }\int_{T/3}^{2T/3} \hat{\bm{\xi}}^{(x,T)}_t(v^*)\;\mathrm{d}t > \frac{T}{6} \right\}.
\]

Now fix some arbitrary initial configuration~$(\mathcal T_0,\mathcal B_0)$ for the grass-bushes-trees process. Note that
\begin{equation}\label{eq_union_bt}
	\mathcal T_t = \bm{\xi}^{\mathcal T_0}_t \quad \text{and}\quad \mathcal B_t = \bm{\xi}^{\mathcal B_0}_t \setminus \bm{\xi}^{\mathcal T_0}_t.
\end{equation}
We claim that
\begin{equation}\label{eq_claim_bushes}
	\{\mathcal T_T \neq \varnothing\} \cap \bigcap_{x \in G} (E_x \cap E_x') \subseteq \{\mathcal T_T \neq \varnothing,\; \mathcal B_T = \varnothing\}.
\end{equation}
To see this, assume that the event on the left-hand side occurs.  Now let~$y \in G$; let us prove that~$y \notin \mathcal B_T$. If~$\bm{\xi}^G_T(y)=0$, then by~\eqref{eq_union_bt} we have~$\mathds{1}\{y \in \mathcal B_T\} \le \bm{\xi}^{\mathcal B_0}_T(y) \le \bm{\xi}^G_T(y)=0$. Assume on the other hand that~$\bm{\xi}^G_T(y)=1$. This is equivalent to~$\hat{\bm{\xi}}^{(y,T)}_T \neq \varnothing$; since~$E_y^\prime$ occurs by assumption, we then have
	\begin{equation}\label{eq_integral1}
\int_{T/3}^{2T/3} \hat{\bm{\xi}}_t^{(y,T)}(v^*)\;\mathrm{d}t > \frac{T}{6}. 
	\end{equation}
Since~$\mathcal T_T \neq \varnothing$ and we have~$\mathcal T_T = \cup_{x \in \mathcal T_0} \bm{\xi}^x_T$, there exists~$x \in \mathcal T_0$ such that~$\bm{\xi}^x_T \neq \varnothing$. Since~$E_x$ occurs by assumption, we have
	\begin{equation}\label{eq_integral2}
	\int_{T/3}^{2T/3} \bm{\xi}^x_t(v^*)\;\mathrm{d}t > \frac{T}{6}.
	\end{equation}
	Now,~\eqref{eq_integral1} and~\eqref{eq_integral2} imply that there is~$t \in [T/3,2T/3]$ such that~$(x,0) \rightsquigarrow (v^*,t) \rightsquigarrow (y,T)$, which gives~$y \in \mathcal T_T$, so~$y \notin \mathcal B_T$. This finishes the proof of~\eqref{eq_claim_bushes}.

From~\eqref{eq_claim_bushes}, it follows that
\[
	\mathbb P(\mathcal T_T \neq \varnothing,\; \mathcal B_T \neq \varnothing) \le \sum_{x\in G} (\mathbb P(E_x^\complement) + \mathbb P((E_x')^\complement)) = 2 \sum_{x\in G} \mathbb P(E_x^\complement),
\]
since~$\mathbb P(E_x) = \mathbb P(E_x')$ by self-duality. In order to bound~$\mathbb P(E_x^\complement)$, let~$S^*$ be the star graph made of~$v^*$ and its neighbours, and define
\[
	\sigma_x:=\inf\{t > 0: S^* \text{ is lit in } \bm{\xi}^x_t \}.
\]
We then bound
\[
	\mathbb P(E_x^\complement) \le \mathbb P\left(\bm{\xi}^x_{T/3} \neq \varnothing,\; \sigma_x > \frac{T}{3}\right) +  \mathbb P\left(\left. \int_{T/3}^{2T/3} \bm{\xi}^x_t(v^*)\;\mathrm{d}t \le \frac{T}{6}\; \right| \;\sigma_x \le T/3\right).
\]
For~$n$ large, apply the integer-interval occupation estimate in the proof of Lemma~\ref{lem_all_stars} after time~$\sigma_x$. On~$\{\sigma_x\le T/3\}$, the complete intervals of the form~$[\sigma_x+k,\sigma_x+k+1]$ contained in~$[T/3,2T/3]$ have total length at least~$T/3-2$. The estimate gives occupation time at least~$\frac34(T/3-2)>T/6$ for~$T>18$. Hence, by the strong Markov property at~$\sigma_x$,
\[
\mathbb P\left(\left. \int_{T/3}^{2T/3} \bm{\xi}^x_t(v^*)\;\mathrm{d}t \le \frac{T}{6}\; \right| \;\sigma_x \le \frac{T}{3}\right) \le \exp(-c_{\mathrm{star}}(\lambda^2 \wedge 1)n^\gamma).
\]
Next, by the assumption on the diameter of the graph and Lemma~\ref{lem_all_stars}, we have
\[\mathbb P(\sigma_x \le t + \lceil C \log n \rceil + 1 \mid \bm{\xi}^x_t \neq \varnothing) \ge  (e^{-1} (1-e^{-\lambda}))^{\lceil C \log n\rceil}\cdot c_{\mathrm{star}}(\lambda \wedge 1)=:q.\]
The expression for~$q$ is a lower bound for the probability that an infection path is formed from anywhere in the graph at time~$t$ to~$v^*$ within~$\lceil C \log n \rceil$ time units, and  that additionally,~$S^*$ becomes lit in an additional time unit. 
Iterating this gives
\[
	\mathbb P\left(\bm{\xi}^x_{T/3} \neq \varnothing,\; \sigma_x > \frac{T}{3}\right) \le (1-q)^{\lfloor (T/3)/(\lceil C \log n \rceil + 1) \rfloor} \le \exp\left\{ - \frac{qT}{6C\log n}\right\}
\]
Recalling that~$T=n^k$, we can now choose~$k$ large enough so that the right-hand side is smaller than~$e^{-n}$ for~$n$ large enough.

We have thus shown that for any~$G \in \mathcal G_n$, and any~$(\mathcal T_0,\mathcal B_0)$, we have
\[
	\mathbb P(\mathcal T_T \neq \varnothing, \; \mathcal B_T \neq \varnothing) \le 2n\left[\exp(-c_{\mathrm{star}}(\lambda^2\wedge 1)n^\gamma)+ \exp(-n)\right].\qedhere
\]
\end{proof}

\section{Coexistence}\label{s_coex}
The goal of this section is to prove Theorem~\ref{thm_coexistence} which, unlike the previous two theorems, shows that the bushes can survive exponentially long. 
Throughout this section, we assume that~$\tau > 2$.

\subsection{Graph conditions}

Given a graph~$G=(V,E)$ and disjoint sets of vertices~$A,B \subset V$, we denote by~$e(A:B)$ the number of edges having one extremity in~$A$ and the other in~$B$.

\begin{lemma}\label{lem_small_complements}
For any~$\epsilon > 0$, there exists~$\eta > 0$ such that with high probability as~$n \to \infty$,~$G_n$ is such that
	\[
		 e\big(A:A^\complement\big) \le \epsilon n \quad \text{for all $A \subset G_n$ with $|A| \le \eta n$}.
	\]
\end{lemma}

\begin{lemma}\label{lem_large_complements}
	For any~$\gamma \in (0,1/2)$, with high probability as~$n \to \infty$,~$G_n$ is such that
	\[
		e\big(A:A^\complement\big) \ge \frac{\gamma n}{50} \quad \text{for all $A \subset G_n$ with $\gamma n \le |A| \le (1-\gamma)n$}.
	\]
\end{lemma}

\subsection{Proof of coexistence}
Let~$\mu:=\sum_m mp(m)$.
\begin{lemma}\label{lem_not_many_trees}
	For any~$\delta \in (0,1)$ and~$\lambda_{\mathcal T} > 0$, the following holds. Let
	\[T:=\exp\{\frac{2\log 2-1}{3}(\delta + 2 \lambda_{\mathcal T}\mu)n\},\] with high probability as~$n \to \infty$,~$G_n$ is such that
	\[
		\inf_{\mathcal T_0:|\mathcal T_0|\le \delta n} \mathbb P(|\mathcal T_t| \le 2(\delta + 2\lambda_{\mathcal T}\mu) n \text{ for all }t \le T) > 1 - 1/T.
	\]
\end{lemma}
\begin{proof}

	We introduce an auxiliary process to bound the trees from above, which we call the \emph{sequoia process} and denote by~$(\bm{\psi}_t)_{t \ge 0}$. This is a very simple process of independent Bernoulli flipping at each site, as follows: 
\begin{itemize}
\item every $v$ with $\bm{\psi}_t(v)=0$ jumps to $\bm{\psi}_t(v)=1$ independently at rate $\lambda_{\mathcal T} \deg(v)$;
\item every $v$ with $\bm{\psi}_t(v)=1$ falls to $\bm{\psi}_t(v)=0$ independently at rate $1$.
\end{itemize}
	We fix~$\mathcal T_0$ with~$|\mathcal T_0| \le \delta n$ and let~$\bm{\psi}_0 =\mathcal T_0$.
Of course, the second jump mechanism above is exactly the usual contact process recovery of the trees, and the first mechanism dominates the infections,  because no vertex can infected by more vertices than it has neighbours. Thus it is clear that the sequoia process dominates the tree process.

We now control the sequoias.
Because $\tau>2$ we can assume by the strong law of large numbers that~$\sum_v D_v \le 2\mu n$, so the number of edges of~$G_n$ satisfies~$|E| \leq \sum D_v \leq 2 \mu n$. We then have
\[
\begin{split}
\frac{\rm d}{{\rm d}t}
	\mathbb E \left[|\bm{\psi}_t|\right]
&=\lambda_{\mathcal T}\sum_v\deg(v)(1-\mathbb P( \bm{\psi}_t(v)=1))-\mathbb E \left[|\bm{\psi}_t|\right] \le 2 \lambda_{\mathcal T}\mu n-\mathbb E \left[|\bm{\psi}_t|\right].
\end{split}
\]
	We solve this differential inequality, given also $|\bm{\psi}_0|\leq\delta n$, to obtain
\[
\mathbb E \left[|\bm{\psi}_t|\right]\le e^{-t}|\bm{\psi}_0|+(1-e^{-t})2\lambda_{\mathcal T}\mu n
\le(\delta+2\lambda_{\mathcal T}\mu) n=:\delta'n.
\]
Then by independence of the vertex occupancy processes and Markov's inequality,
\[
\begin{split}
\mathbb P\left(|\bm{\psi}_t|>2\delta' n\right)
&\le 2^{-2\delta' n}\mathbb E\left[2^{|\bm{\psi}_t|}\right]
=2^{-2\delta' n}\prod_v(1+\mathbb P( \bm{\psi}_t(v)=1))\\
&\le \exp\left(-2\delta' n\log 2+\sum \mathbb P( \bm{\psi}_t(v)=1)\right)
\le \exp\left(-(2\log 2 -1)\delta' n\right),
\end{split}
\]
where in the line break we also used $1+x\le e^x$.

By Fubini's theorem, with~$T$ as in the statement,
\[
\e\left[\int_0^{T+1}\mathds{1}\{|\bm{\psi}_t|>2\delta'n\}\,{\rm d}t\right]
\le \frac{T+1}{T^3}\le\frac{2}{T^2}.
\]
Since the total recovery rate is at most~$n$, after first exceeding~$2\delta'n$ the expected time spent above this threshold during the next unit interval is at least~$1/(2n)$. Thus, by the strong Markov property,
\[
\e\left[\int_0^{T+1}\mathds{1}\{|\bm{\psi}_t|>2\delta'n\}\,{\rm d}t\right]
\ge \frac{1}{2n}\p\left(\sup_{0\le t\le T}|\bm{\psi}_t|>2\delta'n\right).
\]
The probability in the last display is therefore at most~$4n/T^2<1/T$ for all sufficiently large~$n$, proving the lemma.
\end{proof}

\begin{proof}[Proof of Theorem~\ref{thm_coexistence}]
	We fix constants as follows:
	\begin{itemize}
		\item let~$\gamma \in (0,1/2)$; 
		\item fix~$\epsilon \in (0,\gamma/100)$ to be chosen later; 
		\item take~$\eta$ corresponding to~$\epsilon$ in Lemma~\ref{lem_small_complements} (also with~$\eta < \gamma$);
		\item take~$\delta$ and~$\lambda_{\mathcal T}$ small enough that~$2(\delta + 2\lambda_{\mathcal T}\mu) < \eta$. 
	\end{itemize}
Using Lemmas~\ref{lem_small_complements}, Lemma~\ref{lem_large_complements}, and Lemma~\ref{lem_not_many_trees}, with high probability as~$n \to \infty$,~$G_n$ is such that
	\begin{align}\label{eq_renewed_conditions}
		&e\big(A:A^\complement\big) \le \epsilon n \quad \text{$\forall A:|A| \le \eta n$}, \\[.2cm]
		&e\big(A:A^\complement\big) \ge \frac{\gamma n}{100} \quad \forall A:\frac{\gamma}{2} n \le |A| \le \left(1-\frac{\gamma}{2}\right)n,
	\end{align}
	and
	\begin{equation}\label{eq_renewed_tree}
		\inf_{\mathcal T_0:|\mathcal T_0|\le \delta n} \mathbb P(|\mathcal T_t| \le \eta n  \text{ for all }t \le T) > 1 - 1/T,
	\end{equation}
	where~$T:=\exp\{\frac{2\log 2-1}{3}(\delta + 2 \lambda_{\mathcal T}\mu)n\}$. 

	From now on, we fix a realization of~$G_n$ for which all the above is true, and let~$(\mathcal T_t,\mathcal B_t)_t$ be the grass-bushes-tress process on~$G_n$, with~$|\mathcal T_0| < \delta n$ and~$|\mathcal B_0|\ge \gamma n$.

	Given disjoint sets of vertices~$A,B$, we let~$r_+(A,B)$ and~$r_-(A,B)$ be the rates of increase and decrease of the bushes, respectively, when~$(\mathcal T_t,\mathcal B_t)=(A,B)$. Note that, defining
	\[
		\rho_+:=\lambda_{\mathcal B}\left(\frac{\gamma }{100} - \epsilon\right) \quad \text{and}\quad \rho_-:=1+\frac{2\lambda_{\mathcal T}\epsilon}{\gamma},
	\]
	whenever~$|B|\ge\gamma n/2$,~$e(A:A^\complement)\le\epsilon n$ and~$e(B:B^\complement)\ge\gamma n/100$, we have
	\begin{align*}
		r_+(A,B)&=\lambda_{\mathcal B}e\big(B:(A\cup B)^\complement\big)\ge\rho_+ n\ge\rho_+|B|,\\
		r_-(A,B)&=|B|+\lambda_{\mathcal T}e(A:B)\le|B|+\lambda_{\mathcal T}\epsilon n\le\rho_-|B|.
	\end{align*}
	We take~$\epsilon$ small enough and~$\lambda_{\mathcal B}$ large enough that~$\rho_+ > \rho_-$. Note that the above combined with~\eqref{eq_renewed_conditions} and~\eqref{eq_renewed_tree} gives
	\begin{equation}\label{eq_modified_rated}
		|\mathcal T_t| \le \eta n, \; \frac{\gamma}{2}n \le |\mathcal B_t| \le \left(1-\frac{\gamma}{2}\right)n, \quad \Longrightarrow \quad  r_+(\mathcal T_t,\mathcal B_t) \ge  \rho_+ |\mathcal B_t|,\;  r_-(\mathcal T_t,\mathcal B_t) \le \rho_-|\mathcal B_t|.
	\end{equation}

Our plan now is to build on the same probability space as~$(\mathcal T_t,\mathcal B_t)_t$ a birth-and-death chain~$(Y_t)_t$ such that until a certain stopping time, the domination~$Y_t \le |\mathcal B_t|$ holds. We will take~$Y_0 = \lfloor \gamma n \rfloor$, and the jump rates~$q(\cdot,\cdot)$ of~$(Y_t)_t$ will be equal to
	\begin{align*}
		&q(m,m+1)=\rho_+ m,\; m \in \{0,\ldots, \lfloor (1-\gamma/2)n\rfloor - 1\},\\
		&q(m,m-1) = \rho_- m,\;m \in \{0,\ldots,\lfloor (1-\gamma/2)n\rfloor\}.
	\end{align*}
	All other jump rates are zero.

	The coupling is done as follows. At times when~$Y_t \neq |\mathcal B_t|$, we let~$Y_t$ jump independently of~$(\mathcal T_t,\mathcal B_t)$, according to~$q(\cdot,\cdot)$. At times when~$Y_t = |\mathcal B_t|$, we make~$Y_t$ copy the jumps of~$|\mathcal B_t|$ as much as possible, according to the standard scheme: 
	\begin{itemize}
		\item if~$|\mathcal B_t|$ jumps to the right (respectively, to the left),~$Y_t$ flips a coin with probability~$1 \wedge (q(Y_t,Y_t+1)/r_+(\mathcal T_t,\mathcal B_t))$ (respectively,~$1 \wedge (q(Y_t,Y_t-1)/r_-(\mathcal T_t,\mathcal B_t))$) to copy the same jump;
		\item with an additional rate~$0 \vee (q(Y_t,Y_t+1)-r_+(\mathcal T_t,\mathcal B_t))$ (respectively,~$0 \vee (q(Y_t,Y_t-1)-r_-(\mathcal T_t,\mathcal B_t))$),~$Y_t$ jumps on its own to the right (respectively, to the left).
	\end{itemize}
	It should now be clear that~$(Y_t)_t$ jumps with the correct rates.

	Define the stopping time
	\[
		\sigma:= \inf\left\{t \ge 0: |\mathcal T_t| > \eta n \text{ or } Y_t < \frac{\gamma}{2} n\right\}.
	\]
	We claim that we have
	\begin{equation}\label{eq_key_YB}Y_t \le |\mathcal B_t|\quad \text{for all }t \in [0,\sigma].\end{equation}
		The inequality holds at~$t=0$, since we take~$Y_0=\lfloor \gamma n \rfloor \le |\mathcal B_0|$. Moreover, assume that for some~$s \le \sigma$ we have~$Y_{s-} \le |\mathcal B_{s-}|$, and suppose a jump of~$(\mathcal B_t,Y_t)$ occurs at time~$s$. If~$Y_{s-}<|\mathcal B_{s-}|$, then we clearly have~$Y_{s}\le |\mathcal B_s|$. Assume on the other hand that~$Y_{s-}=|\mathcal B_{s-}|$. Since~$s-<\sigma$, we have~$|\mathcal T_{t-}|\le\eta n$ and~$Y_{s-} \ge \frac{\gamma}{2}n$. We also have~$Y_{s-} \le \lfloor (1-\gamma/2)n \rfloor$, since~$(Y_t)_t$ never goes above this value. Hence,~$\frac{\gamma}{2}n \le Y_{s-}= \mathcal{B}_{s-} \le \lfloor (1-\gamma/2)n \rfloor$. Then, by~\eqref{eq_modified_rated}, we have~$ r_+(\mathcal T_t,\mathcal B_t) \ge  \rho_+ |\mathcal B_t|$ and~$ r_-(\mathcal T_t,\mathcal B_t) \le \rho_-|\mathcal B_t|$. Confronting this with the jump mechanism of~$(Y_t)_t$, we see that~$Y_t \le |\mathcal B_t|$.  This completes the proof of~\eqref{eq_key_YB}.

		Now, the definition of~$\sigma$ and~\eqref{eq_key_YB} give
	\begin{equation}\label{eq_ineq_sigma}
		\mathbb P(\mathcal B_t = \varnothing) \le \mathbb P(\sigma \le t) \le \mathbb P(\exists s \le t: |\mathcal T_s| > \eta n) + \mathbb P(\exists s \le t: Y_s < \gamma n/2).
	\end{equation}

Using~$\rho_+>\rho_-$,~$|Y_0|=\lfloor \gamma n\rfloor$ and a large deviations bound for random walks, we can find a constant~$c_0>0$ such that
	\begin{equation}
	\label{eq_reneww2}
		\mathbb P\left(\exists s \le \exp\{c_0 n\}: Y_s < \gamma n /2 \right)<\exp\{-c_0 n\}.
	\end{equation}
	Recall that~$T:=\exp\{\frac{2\log 2-1}{3}(\delta + 2 \lambda_{\mathcal T}\mu)n\}$. Taking~$c_1:=c_0 \wedge \frac{2\log 2-1}{3}(\delta + 2 \lambda_{\mathcal T}\mu)$ and~$T':=\exp\{c_1 n\}$, we have
\begin{equation}
	\label{eq_reneww3}
	\mathbb P(\exists s \le T': |\mathcal T_s| > \eta n) \le \mathbb P(\exists s \le T: |\mathcal T_s| > \eta n) \stackrel{\eqref{eq_renewed_tree}}{\le} 1/T \le \exp\{-c_1 n\}.
\end{equation}
Now, using~\eqref{eq_reneww2} and~\eqref{eq_reneww3} in~\eqref{eq_ineq_sigma} (with~$t=T'$) completes the proof.
\end{proof}

\section{Proof of graph conditions} \label{s_proof_graph}
\label{s_graph_properties}

In this final section, we have deferred the proofs of results on graph properties that were previously needed. First we have a result about the order statistics of the i.i.d. stub-counts that will build the graph and be useful throughout, and then the other results are grouped by which section they are needed for.

\subsection{Stub-count bounds}
For now, we allow for any~$\tau > 1$; in subsequent subsections, we will specialize to~$\tau \in (1,2)$ and to~$\tau \ge 2$.
Define
\begin{equation*}
	\bar{F}(x):=p(\{m: m > x\}),\quad x \in [0,\infty).
\end{equation*}
It follows from~\eqref{eq_condition_stubcount} that there exist~$c_1, c_2 > 0$ such that
\begin{equation}\label{eq_tail_constants}
	c_1 < \frac{\bar{F}(x)}{x^{-(\tau-1)}} < c_2 \quad \text{ for all } x \in [1,\infty).
\end{equation}

Let~$D_1,\ldots, D_n$ be independent, all distributed as~$p$.
%
We let~$(D_{(1)},\ldots,D_{(n)})$ be the sequence~$(D_1,\ldots,D_n)$ arranged in non-increasing order.

\begin{lemma}\label{lem_event_An}
	There exist~$\alpha_1,\alpha_2 > 0$ such that defining the event
	\begin{equation}\label{eq_def_event_An}\begin{split}
	\mathcal A_n := &\bigcap_{i=1}^{\lfloor (\log n)^2 \rfloor}\left\{ (\log n)^{-\frac{3}{\tau-1}} \le \frac{D_{(i)}}{(n/i)^{1/(\tau-1)}} \le (\log n)^{\frac{3}{\tau-1}} \right\}\\
	&\qquad \cap \bigcap_{i=\lfloor (\log n)^2 \rfloor+1}^n \left\{\alpha_1 \le \frac{D_{(i)}}{(n/i)^{1/(\tau-1)}} \le \alpha_2\right\},
\end{split}
\end{equation}
we have~$\mathbb P(\mathcal A_n) \to 1$ as~$n \to \infty$.
\end{lemma}
\begin{proof}
	We will use the Chernoff bounds, for~$X \sim\mathrm{Bin}(N,q)$:
	\begin{align}
		\label{eq_chernoff_upper}		&\mathbb P(X \le Nq - t) \le \exp \left\{-\frac{t^2}{2Nq} \right\},\quad t \ge 0,\\
		\label{eq_chernoff_lower}		& \mathbb P(X \ge Nq + t) \le \exp \left\{-\frac{t^2}{2(Nq + t/3)} \right\},\quad t \ge 0.
	\end{align}

	For~$i \in [n]$ and~$\alpha > 0$, write~$m_\alpha(i):=\alpha (n/i)^{1/(\tau-1)}$. Note that~\eqref{eq_tail_constants} gives
	\begin{equation}\label{eq_bounds_alpha}
		c_1 \alpha^{-(\tau-1)}i \le n \bar{F}(m_\alpha(i)) \le c_2 \alpha^{-(\tau-1)}i \quad \text{whenever } m_\alpha(i) \ge 1.
	\end{equation}

For~$i \in [n]$ and~$\alpha_1 > 0$, we have
\begin{equation*}
	\mathbb P(D_{(i)} \le m_{\alpha_1}(i)) = \mathbb P (|\{j:D_j > m_{\alpha_1}(i)\}| < i) = \mathbb P( \mathrm{Bin}(n,\bar{F}(m_{\alpha_1}(i)))<i).
\end{equation*}
	Take~$\alpha_1$ small enough that~$c_1\alpha_1^{-(\tau-1)} > 1$. In case~$m_{\alpha_1}(i) < 1$, we have $\bar{F}(m_{\alpha_1}(i))=1$ (since degrees are at least~$3$), so $\mathbb P( \mathrm{Bin}(n,\bar{F}(m_{\alpha_1}(i)))<i)=0$. Now assume that~$m_{\alpha_1}(i) \ge 1$; then, the definition of~$\alpha_1$ and~\eqref{eq_bounds_alpha} give~$n \bar{F}(m_{\alpha_1}(i)) > i$. Then,~\eqref{eq_chernoff_upper}  gives
	\begin{equation*}
		\mathbb P( \mathrm{Bin}(n,\bar{F}(m_{\alpha_1}(i)))\le i) \le \exp \left\{- \frac{(c_1\alpha_1^{-(\tau-1)}i - i)^2}{2c_2 \alpha_1^{-(\tau-1)}i} \right\}.
	\end{equation*}
	We thus have the bound
	\begin{equation}\label{eq_stubs_below}
		\mathbb P\left(D_{(i)} \le \alpha_1 \left(\frac{n}{i} \right)^{1/(\tau-1)}\right)\le \exp \left\{- \frac{(c_1\alpha_1^{-(\tau-1)} - 1 )^2}{2c_2 \alpha_1^{-(\tau-1)}} \cdot i \right\}.
	\end{equation}

Next, for~$i \in [n]$ and~$\alpha_2 > 0$, we have
\begin{equation*}
	\mathbb P(D_{(i)} > m_{\alpha_2}(i)) = \mathbb P (|\{j:D_j > m_{\alpha_2}(i)\}| \ge i) = \mathbb P( \mathrm{Bin}(n,\bar{F}(m_{\alpha_2}(i))) \ge i).
\end{equation*}
	Now take~$\alpha_2 \ge 1$ large enough that~$c_2\alpha_2^{-(\tau-1)}<1$. Since~$m_{\alpha_2}(i) \ge 1$, we can again apply~\eqref{eq_bounds_alpha}, which gives $n\bar{F}(m_{\alpha_2}(i)) < i$. Then, also bounding~$i-n\bar{F}(m_{\alpha_2}(i)) \le i$,~\eqref{eq_chernoff_lower} gives
	\begin{equation*}
		\mathbb P( \mathrm{Bin}(n,\bar{F}(m_{\alpha_2}(i))) \ge i) \le \exp\left\{- \frac{(i - c_2\alpha_2^{-(\tau-1)}i)^2}{2(c_2\alpha_2^{-(\tau-1)}i + i/3)}\right\}.
	\end{equation*}
	We then have the upper bound
	\begin{equation}\label{eq_stubs_above}
		\mathbb P\left(D_{(i)} > \alpha_2\left(\frac{n}{i}\right)^{1/(\tau-1)}\right) \le \exp\left\{ -\frac{(1-c_2\alpha_2^{-(\tau-1)})^2}{2(c_2\alpha_2^{-(\tau-1)}+1/3)} \cdot  i \right\}.
	\end{equation}

	Now,~\eqref{eq_stubs_below},~\eqref{eq_stubs_above} and a union bound imply that
	\begin{equation}\label{eq_first_convergence_stub}
		\mathbb P\left(\bigcap_{i=\lfloor (\log n)^2\rfloor + 1}^{n}\left\{\alpha_1 \le \frac{D_{(i)}}{(n/i)^{1/(\tau-1)}} \le \alpha_2  \right\}\right) \xrightarrow{n \to \infty} 1.
	\end{equation}

	We now turn to the remaining part in the event~$\mathcal A_n$. Abbreviate~${i}^*:=\lfloor (\log n)^2 \rfloor$.
Assume that~$n$ is large enough that
	\[
		\alpha_1 \left( \frac{n}{{i}^*+1}\right)^{1/(\tau-1)} \ge \left(\frac{n}{(\log n)^3}\right)^{1/(\tau-1)},
	\]
	where~$\alpha_1$ is the constant of~\eqref{eq_first_convergence_stub}. 
	In the event~$\{D_{(i^*+1)} \ge \alpha_1(n/(i^*+1))^{1/(\tau-1)}\}$, for all~$i \in \{1,\ldots, i^*\}$ we then have
	\[
		D_{(i)} \ge D_{(i^*+1)} \ge  \left(\frac{n}{(\log n)^3}\right)^{1/(\tau-1)} \ge {(\log n)^{-3/(\tau-1)}} \cdot \left( \frac{n}{i} \right)^{1/(\tau-1)}.
	\]
This proves that
	\[
\mathbb P\left(  
	\bigcap_{i=1}^{i^*} \left\{ D_{(i)} \ge (\log n)^{-3/(\tau-1)} \cdot \left( \frac{n}{i} \right)^{1/(\tau-1)} 
\right\}
\right) \xrightarrow{n \to \infty} 1.
	\]

	Finally, recalling that~$D_{(1)}=\max_{i \in [n]} D_i$, we have
	\begin{align*}
		\mathbb P\left( D_{(1)} > (n\log n)^{1/(\tau-1)}\right) &\le n \cdot \bar{F}((n\log n)^{1/(\tau-1)})\\
		&\stackrel{\eqref{eq_tail_constants}}{\le}c_2 (\log n)^{-1}  \xrightarrow{n \to \infty} 0.
	\end{align*}
On the event that~$D_{(1)} \le (n\log n)^{1/(\tau-1)}$, for all~$i \in \{1,\ldots, i^*\}$ we have
\[
	D_{(i)} \le (n\log n)^{\frac{1}{\tau-1}} = (\log n)^{\frac{3}{\tau-1}} \cdot \left( \frac{n}{(\log n)^2} \right)^{\frac{1}{\tau-1}} \le (\log n)^{\frac{3}{\tau-1}}\cdot \left(\frac{n}{i} \right)^{\frac{1}{\tau-1}}.
\]
This proves that
	\[
\mathbb P\left(  
	\bigcap_{i=1}^{i^*} \left\{ D_{(i)} \le (\log n)^{\frac{3}{\tau-1}} \cdot \left( \frac{n}{i} \right)^{1/(\tau-1)} 
\right\}
\right) \xrightarrow{n \to \infty} 1. \qedhere
	\]

\end{proof}

\subsection{Proofs of statements from Section~\ref{s_no_coex}}
\label{s_proof_graph_12}

As in Section~\ref{s_no_coex}, we assume here that~$\tau \in (1,2)$.

\begin{proof}[Proof of Lemma~\ref{lem_event_Bn}]
	For these graph lemmas and their proofs, label the vertices so that vertex~$i$ has stub-count~$D_{(i)}$, breaking ties arbitrarily. We write~$L_n := \sum_{i \in [n]} D_{(i)}$.

	Letting~$\mathcal A_n$ be as in~\eqref{eq_def_event_An}, we will prove that
	\begin{equation*}
		\mathbb P(\mathcal B_n \mid \mathcal A_n) \xrightarrow{n \to \infty} 1.
	\end{equation*}
	The desired result will then follow from Lemma~\ref{lem_event_An}.

	Let~$\alpha_1,\alpha_2$ be the constants from Lemma~\ref{lem_event_An}. Assume that~$n$ is large enough that~$ \alpha_1> (\log n)^{-3/(\tau-1)}$ and~$\alpha_2 < (\log n)^{3/(\tau-1)}$.
	Condition on a stub count sequence~$(D_i)_{i \in [n]}$ for which the event~$\mathcal A_n$ occurs.

	Assume that at first, we reveal all degrees in the graph, but none of the matchings. We follow an exploration algorithm to reveal the neighbours of~$i$. We proceed in steps, which are labelled by an integer~$\ell \ge 1$. Let~$\mathcal S_0 = \varnothing$. The set~$\mathcal S_\ell$ will denote the set of stubs of the graph that have been touched by the algorithm up to (and including) the~$\ell$th step. The algorithm stops after a step~$\ell$ in which either all stubs of~$i$ are in~$\mathcal S_\ell$, or a stub of~$i$ has been matched to a stub of~$j$.

	Step~$\ell$ is defined as follows: we take a stub~$\mathsf s$ of~$i$ not belonging to~$\mathcal S_{\ell-1}$, and match it to a stub~$\mathsf s'$ chosen uniformly at random from the set of all stubs, excluding~$\mathcal S_{\ell-1} \cup \{\mathsf s\}$. We then set~$\mathcal S_\ell := \mathcal S_{\ell-1} \cup \{\mathsf s, \mathsf s'\}$.

	Given that the algorithm has not halted by step~$\ell$, the probability that a stub of~$i$ is matched to a stub of~$j$ in step~$\ell$ is at least~$ \frac{D_{(j)}}{L_n-2(\ell-1)}$. Since~$\mathcal A_n$ occurs, we can bound~$D_{(k)} \ge (\log n)^{-3/(\tau-1)} (n/k)^{1/(\tau-1)}$ for all~$k \in [n]$, and
	\[L_n:=\sum_{k \in [n]} D_k \le (\log n)^{3/(\tau-1)} \sum_{k=1}^n (n/k)^{1/(\tau-1)} \le C(\log n)^{3/(\tau-1)} n^{1/(\tau-1)} .\] Hence, the probability that no stub of~$i$ is matched to a stub of~$j$ is at most
	\begin{align*}
		\left(1- \frac{D_{(j)}}{L_n} \right)^{D_{(i)}/2} &\le \exp\left\{ - \frac{D_{(i)}D_{(j)}}{2L_n} \right\} \\
		&\le \exp \left\{ - \frac{(\log n)^{-6/(\tau-1)}(n^2/(ij))^{1/(\tau-1)}}{2C(\log n)^{3/(\tau-1)}n^{1/(\tau-1)}} \right\}.
	\end{align*}
	In case~$ij \le 4n/(\log n)^{10}$, the right-hand side above is smaller than
	\[
		\exp\left\{-\frac{1}{2^{1+\frac{2}{\tau-1}}C} (\log n)^{\frac{1}{\tau-1}} \right\}.
	\]
	Taking a union bound over all possible pairs~$i,j$  completes the proof.
\end{proof}

\begin{corollary}\label{cor_lower_bound}
With high probability as~$n \to \infty$, we have
	\[
		\deg(i) \ge \max\left\{1,\; {(\log n)^{-10}}\cdot \frac{n}{i}\right\} \quad \text{for all } i \in [n].
	\]
\end{corollary}
\begin{proof}
	By Lemmas~\ref{lem_event_An},~\ref{lem_connectedness} and~\ref{lem_event_Bn}, we have that with high probability as~$n \to \infty$, $G_n$ is connected and the event~$\mathcal B_n$ occurs. Assume that this is the case.

	Connectedness implies that every vertex has at least one neighbour, that is,~$\deg(i) \ge 1$ for all~$i$. Assume that~$i$ is such that~$(\log n)^{-10} n/i \ge 1$. The assumption that~$\mathcal B_n$ occurs implies that~$i$ is connected to all vertices~$j$ with~$1 \le j \le \lfloor 4(n/i)(\log n)^{-10}\rfloor$ (except for~$i$ itself), so~$\deg(i) \ge \lfloor 4(n/i)(\log n)^{-10}\rfloor -1 \ge (n/i) (\log n)^{-10}$.
\end{proof}

\begin{lemma}\label{lem_upper_degrees}
There exists~$\mathsf{k} > 0$ such that, with high probability as~$n \to \infty$, we have
	\[
		\deg(i) \le (\log n)^{\mathsf{k}}\cdot \frac{n}{i} \quad \text{for all } i \in [n].
	\]
\end{lemma}
\begin{proof}
	By Lemma~\ref{lem_event_An},  it suffices to prove that
	\begin{equation*}
		\mathbb P\left(\left.\deg(i) \le (\log n)^{\mathsf{k}}\cdot \frac{n}{i} \text{ for all } i \in [n] \;\right|\; \mathcal A_n\right) \xrightarrow{n \to \infty} 1
	\end{equation*}
	for some appropriately chosen~$\mathsf{k}$.

	Fix a realization of the stub count sequence~$(D_i)_{i \in [n]}$ for which~$\mathcal A_n$ occurs.  For each~$i \in [n]$, define
	\[
		j^*(i):= \lfloor 2(\log n)^{-10}  (n/i) \rfloor.
	\]
	We will split the upper bound on the degree of~$i$ into three cases.

	\emph{First}, assume that~$\frac{2n}{3(\log n)^{10}} < i \le n$, so that~$j^*(i) \le 2$. In this case, we bound
	\begin{align*}
		\deg(i) \le D_{(i)} \le D_{(\lfloor 2n/(3(\log n)^{10})\rfloor)} &\le (\log n)^{\frac{3}{\tau-1}}\cdot \left( \frac{n}{\lfloor 2n/(3(\log n)^{10})\rfloor} \right)^{\frac{1}{\tau-1}}\\
		&\hspace{-2cm}\le (\log n)^{\frac{3}{\tau-1}}\cdot \left((\log n)^{11}\right)^{\frac{1}{\tau-1}} \le (\log n)^{14/(\tau-1)} \cdot \frac{n}{i},
	\end{align*}
	since~$\frac{n}{i} \ge 1$.

	\emph{Second}, assume that~$i \in \{1,2\}$. In this case, we use the trivial deterministic bounds~$\deg(1) \le n$ and~$\deg(2) \le n$.

	\emph{Third,} assume that~$3 \le i \le \frac{2n}{3(\log n)^{10}}$.
We bound
	\begin{align*}
		\deg(i) &\le j^*(i) + |\{j > j^*(i):\; i \sim j\}| \le \frac{n}{i} + |\{j > j^*(i):\; i \sim j\}|.
	\end{align*}
	In order to bound the second term on the right-hand side, we proceed as in the proof of Lemma~\ref{lem_event_Bn}, and pair the stubs of~$i$ one by one. Define
	\begin{equation}\label{eq_notation_Lnj}
	L_n(j):= \sum_{k=j}^n D_{(k)}\quad \text{and} \quad L_n:=L_n(1)=\sum_{\ell \in [n]} D_{(\ell)}.
	\end{equation}
	 The number of stubs of~$i$ that get matched to stubs of vertices~$j$ with~$j > j^*(i)$ is stochastically dominated by a random variable~$Z_i$ with distribution
	\begin{equation}
		Z_i \sim \mathrm{Bin}\left(D_{(i)},\frac{L_n(j^*(i))}{L_n - 2D_{(i)}} \right).
	\end{equation}
	Recall that, since~$\mathcal A_n$ occurs, we have
	\begin{equation} \label{eq_recall_Di} D_{(i)} \le (\log n)^{3/(\tau-1)}\cdot (n/i)^{1/(\tau-1)}.\end{equation}
		For all~$j \ge 2$ we bound
	\begin{equation}\label{eq_upper_bound_Lnj}
		\begin{split}
		&L_n(j) \le \sum_{k=j}^n (\log n)^{\frac{3}{\tau-1}}\cdot \left( \frac{n}{k}\right)^{\frac{1}{\tau-1}}\\
		&\le ((\log n)^3 n)^{\frac{1}{\tau-1}} \int_{j-1}^\infty t^{-\frac{1}{\tau-1}}\;\mathrm{d}t = \frac{\tau-1}{2-\tau}((\log n)^3 n)^{\frac{1}{\tau-1}} \cdot (j-1)^{-\frac{2-\tau}{\tau-1}}.
		\end{split}
	\end{equation}
	Bounding~$j^*(i)-1 \ge \frac{1}{2(\log n)^{10}}\frac{n}{i}$, we obtain
	\begin{equation}\label{eq_bound_Sstar}
		L_n(j^*(i)) \le (\log n)^{\mathsf{k}_0}n i^{\frac{2-\tau}{\tau-1}}.
	\end{equation}
	where~$\mathsf{k}_0$ only depends on~$\tau$. 

	Next, since~$i \ge 3$, we have
	\[L_n - 2D_{(i)} \ge L_n - D_{(i)} - D_{(i-1)} = \sum_{k \in [n]\backslash \{i-1,i\}} D_{(k)} \ge  D_{(1)};\] again using the assumption that~$\mathcal A_n$ occurs, we obtain
	\begin{equation}\label{eq_bound_Ln}
		L_n - 2D_{(i)} \ge (\log n)^{-\frac{3}{\tau-1}} \cdot n^{\frac{1}{\tau-1}}.
	\end{equation}
	
	Now, putting together~\eqref{eq_recall_Di},~\eqref{eq_bound_Sstar} and~\eqref{eq_bound_Ln}, there is a constant~$\mathsf{k}_1 > 0$ depending only on~$\tau$ such that
	\[
		\mathbb E[Z_i] \le (\log n)^{\mathsf{k}_1} \cdot \frac{n}{i}.
	\]
	We use the bound, for~$X \sim \mathrm{Bin}(N,q)$:
\begin{equation}
	\label{eq_other_chernoff}
		\mathbb P(X \ge x) \le \left( \frac{eNq}{x} \right)^x,\quad x > Nq.
\end{equation}
	This gives
	\[
		\mathbb P(Z_i >  (\log n)^{\mathsf{k}_1+1} \cdot n/i \mid (D_k)_{k \in [n]}) \le (\log n/e)^{-(\log n)^{\mathsf{k}_1+1} n/i}.
	\]

	We have now concluded the three cases. Using the bounds we obtained in each case allows for a union bound over~$i \in [n]$ that shows that~$\deg(i) \le (\log n)^{\mathsf{k}} n/i$ with high probability as~$n \to \infty$, for some~$\mathsf{k} > 0$ depending only on~$\tau$.
\end{proof}

\begin{proof}[Proof of Lemma~\ref{lem_bounds_degrees}]
	The statement follows directly from Corollary~\ref{cor_lower_bound} and Lemma~\ref{lem_upper_degrees}.
\end{proof}

\begin{proof}[Proof of Lemma~\ref{lem_degrees_into_sats}]
	For each~$i \in [n]$, we write
	\[
		N_\epsilon(i) := |\{j: n^{1-\epsilon} \le j \le n,\; i \sim j \}|.	\]
We will prove that
	\[\mathbb P\left( \left. N_\epsilon(i) \le \log n +  (\log n)^{\mathsf{k}'} \cdot \left(\frac{n^\epsilon}{i}\right)^{\frac{2-\tau}{\tau-1}} \cdot \frac{n}{i}     \;\;\forall  i \in [n],\; i \ge n^{\epsilon} \; \right| \; \mathcal A_n  \right) \xrightarrow{n \to \infty} 1,  \]
	for some appropriately chosen~$\mathsf{k}'$;
	together with Lemma~\ref{lem_event_An}, this will conclude the proof. 

	Fix a realization~$(D_i)_{i \in [n]}$ of the stub count sequence for which the event~$\mathcal A_n$ occurs.
	Fix~$i$ with~$n^\epsilon \le i \le n$. By matching the stubs of~$i$ one by one as in the proof of Lemma~\ref{lem_event_Bn}, we see that~$N_\epsilon(i)$ is stochastically dominated by a random variable~$Z_i'$ with
	\[
		Z_i' \sim \mathrm{Bin}\left(D_{(i)}, \frac{L_n(j^*)}{L_n - 2D_{(i)}}\right),
	\]
	where~$j^*:= \lfloor n^{1-\epsilon}\rfloor$ and~$L_n(j^*)$ and~$L_n$ are as in~\eqref{eq_notation_Lnj}. Bounding as in~\eqref{eq_upper_bound_Lnj}, we have
	\begin{equation}\label{eq_Lnjstar}
		L_n(j^*) \le (\log n)^{\mathsf{k}_0'}\cdot n^{1+\epsilon \frac{2-\tau}{\tau-1}},
	\end{equation}
	where~$\mathsf{k}_0'$ only depends on~$\tau$.
	Also using~\eqref{eq_recall_Di} and~\eqref{eq_bound_Ln} gives
	\[
		\mathbb E[Z_i'] \le (\log n)^{\mathsf{k}_1'} \cdot \frac{n}{i} \left( \frac{n^\epsilon}{i} \right)^{\frac{2-\tau}{\tau-1}}
	\]
	for some~$\mathsf{k}_1'$.
	Using the bound~\eqref{eq_other_chernoff} with~$x=x_i:=\log n \cdot (1 + \mathbb E[Z_i'])$, we obtain
	\begin{align*}
		\mathbb P\left( N_\epsilon(i) > x_i \mid (D_k)_{k \in [n]} \right)  &\le \mathbb P\left(Z_i' > x_i \right) \le \left(\frac{e \mathbb E[Z_i']}{x_i} \right)^{x_i}. 
	\end{align*}
	Bounding~$x_i \ge \log n\cdot \mathbb E[Z_i']$ in the denominator and~$x_i \ge \log n$ in the exponent, the right-hand side is at most
	\[
		\left( \frac{\log n}{e}\right)^{-\log n} = n^{-\log \log n + 1}
	\]
	We then have~$n \cdot \mathbb P(Z_i' > x_i) \to 0$ as~$n \to \infty$, so a union bound over~$i$ completes the proof.	
\end{proof}

\begin{proof}[Proof of Lemma~\ref{lem_sizes_satellites}]
	Fix~$\epsilon > 0$ to be taken small and~$m \in \mathbb N$ to be taken large, depending on~$\epsilon$.

	We will prove that the conditional probability of the event described in the statement of the lemma given~$\mathcal A_n$ tends to~$1$ as~$n \to \infty$. 

	Fix a realization~$(D_i)_{i \in [n]}$ of the stub count sequence for which~$\mathcal A_n$ occurs. As before, write~$j^*:=\lfloor n^{1-\epsilon} \rfloor$. Let~$i$ be a vertex with~$j^* \le i \le n$. Since~$\mathcal A_n$ occurs, all stub counts in~$\{j^*,\ldots,n\}$ are at most
	\[
		M:=\left\lceil(\log n)^{3/(\tau-1)}\cdot(n/j^*)^{1/(\tau-1)}\right\rceil.
	\]

	We explore the connected component of~$i$ in the subgraph induced by~$\{j: j^* \le j \le n\}$ in a breadth-first manner. This means that we form edges by matching stubs one by one (similarly to what was done in the proof of Lemma~\ref{lem_event_Bn}), ignoring edges that leave~$\{j: j^* \le j \le n\}$. We stop the exploration either when we run out of stubs to reveal (due to having exhausted a connected component), or when we have discovered~$m+1$ vertices, including~$i$. Before this stopping time, every stub being explored belongs to one of at most~$m$ discovered vertices, so at most~$mM$ matchings are performed.

	Assume that the exploration has run up to a certain step~$\ell < mM$, and there is at least one more stub to be explored next. The probability that this stub is matched to a vertex in~$\{j^*,\ldots, n\}$ is at most
	\[
		q:=\frac{L_n(j^*)}{L_n-2mM},
	\]
	where~$L_n(j^*)$ and~$L_n$ are as in~\eqref{eq_notation_Lnj}. 

	The subgraph that we reveal in the exploration can be made into a tree if, when forming a new edge, we discard it in case it joins two already discovered vertices.
	Since every vertex has at most~$M$ stubs and each matching discovers a new vertex with conditional probability at most~$q$, the exploration up to this stopping time is stochastically dominated by a Galton--Watson tree~$\mathcal T$ with offspring distribution~$X\sim\mathrm{Bin}(M,q)$.
	In particular, the probability that the exploration finds more than~$m$ vertices is at most~$ \mathbb P(|\mathcal T| > m)$.
	
	To bound this probability, we use the random walk encoding of the Galton--Watson tree. Specifically, we let~$X_1,X_2,\ldots$ be independent and identically distributed, with~$X_i \sim \mathrm{Bin}(M,q)$, and then we have
	\[
		\mathbb P(|\mathcal T| > m) \le \mathbb P(X_1+\cdots+X_m \ge m). 
	\]
	Note that~$X_1+\cdots+X_m \sim \mathrm{Bin}(Mm,q)$. In order to give an upper bound on this probability, we bound~$M$ and~$q$.
	\[
		M = \left\lceil(\log n)^{\frac{3}{\tau-1}}\left( \frac{n}{\lfloor n^{1-\epsilon} \rfloor}\right)^{\frac{1}{\tau-1}}\right\rceil \le 1+2^{\frac{1}{\tau-1}}(\log n)^{\frac{3}{\tau-1}} \cdot n^{\frac{\epsilon}{\tau-1}}.
	\]
	To bound~$q$, we use the upper bound~\eqref{eq_Lnjstar} for~$L_n(j^*)$, and the lower bound (which holds for~$n$ large enough, depending on~$m$, since~$mM=o(D_{(1)})$ for~$\epsilon<1$):
	\[
		L_n \ge D_{(1)} \quad \Longrightarrow \quad L_n - 2mM \ge \frac12 D_{(1)} \ge \frac{1}{2(\log n)^{\frac{3}{\tau-1}}} \cdot n^{\frac{1}{\tau-1}}.
	\]
	These bounds give
	\[
		q = \frac{L_n(j^*)}{L_n-2mM} \le  (\log n)^{\mathsf{k}_2}\cdot n^{-(1-\epsilon) \frac{2-\tau}{\tau-1}},
	\]
	where~$\mathsf{k}_2$ depends only on~$\tau$. Hence,
	\[
		Mq \le (\log n)^{\mathsf{k}_2'} \cdot n^{-(1-\epsilon) \frac{2-\tau}{\tau-1}+\frac{\epsilon}{\tau-1}}.
	\]
	Since~$\tau \in (1,2)$, we can choose~$\epsilon$ small enough that~$(1-\epsilon) \frac{2-\tau}{\tau-1}-\frac{\epsilon}{\tau-1} > 0$.

	Using the bound~\eqref{eq_other_chernoff}, we have
	\begin{align*}
		&\mathbb P(X_1 + \cdots +X_m \ge m) \le (eMq)^{m} = \exp\left\{ \log(eMq)\cdot m \right\}\\
		&\qquad \le \exp \left\{\left(1+ \mathsf{k}_2' \log \log n - \left((1-\epsilon) \frac{2-\tau}{\tau-1}-\frac{\epsilon}{\tau-1} \right)\log n\right)\cdot m \right\}.
	\end{align*}
	Taking~$m$ large enough and~$n$ large enough, this is smaller than~$\exp\{-2\log n\}=n^{-2}$.

	We have thus proved that the probability that~$i$ belongs to a connected component with more than~$m$ vertices is~$o(1/n)$. A union bound over~$i$ concludes the proof.
\end{proof}

\subsection{Proofs of statements from Section~\ref{s_second_no_coex}}
\label{s_proof_graph_2}

As in Section~\ref{s_second_no_coex}, here we assume that~$\tau \ge 2$.

	We start proving Lemma~\ref{lem_diameter}. When $\tau>2$, the statement is an immediate consequence of~\cite[Theorem 7.19]{remco}. However, since we also need to cover~$\tau = 2$, we give a different proof (which in fact covers all~$\tau > 1$), taking advantage of the fact that our law~$p$ is supported on integers larger than or equal to~$3$. 

	We will use the following lemma. Given a deterministic sequence of positive integers~$\mathbf{d}=(d_1,\ldots,d_n)$, let~$G^{\mathbf{d}}_n$ be the configuration model constructed using~$\mathbf{d}$ as the stub-count sequence.
	\begin{lemma}\label{lem_new_cm}
		There exists~$C_0 > 0$ such that
		\[
			\inf_{\mathbf{d}} \mathbb P(\mathrm{diam}(G^{\mathbf{d}}_n) \le C_0\log n) \xrightarrow{n \to \infty} 1,
		\]
		where the infimum is taken over all sequences~$\mathbf d=(d_1,\ldots,d_n)$ with~$d_i \in \{3,4,5\}$ for all~$i$ and~$\sum_i d_i$ even.
	\end{lemma}
	This can be proved with the standard methods of bounding the diameter of the configuration model; see for instance~\cite[Section~7]{remco}.

\begin{proof}[Proof of Lemma~\ref{lem_diameter}]
We note that the graph diameter is not affected by the deletion of self-loops and of parallel edges. Hence, in this proof we may work with the standard, non-erased configuration model.

	Condition on the degree sequence~$(D_1,\ldots,D_n)$; assume that its decreasing ordering~$(D_{(1)},\ldots,D_{(n)})$ satisfies the event~$\mathcal A_n$ from~\eqref{eq_def_event_An}. For each~$i$, we can find (unique) $a_i \in \mathbb N_0$ and~$b_i \in \{3,4,5\}$ such that~$D_i=3a_i+b_i$, with~$a_i \in \mathbb N_0$ and~$b_i \in \{3,4,5\}$.

	We now construct an auxiliary graph~$\widehat{G}_n$. The vertex set is
	\[\widehat{V}_n:=\{(i,r):i \in [n],\; r \in [a_{i}+1]\}.\]
	For each~$i$, we endow vertices~$(i,1),\ldots,(i,a_{i})$ with~$3$ stubs, and we endow vertex~$(i,a_{i}+1)$ with~$b_{i}$ stubs. We then construct the set of edges by taking a perfect matching of all stubs in this graph. Note that if for each~$i$ we glue together the vertices~$(i,1),\ldots,(i,a_i+1)$, we recover the original random graph~$G_n$. Moreover, it is readily seen that~$\mathrm{diam}(\widehat{G}_n) \ge \mathrm{diam}(G_n)$.

	Now, by Lemma~\ref{lem_new_cm}, the diameter of~$\widehat{G}_n$ is at most~$C_0 \log(|\widehat{V}_n|)$ with high probability. In order to bound~$|\widehat{V}_n|$, we first note that~$a_i \le 3D_i$ for each~$i$. Using the assumption that the event~$\mathcal A_n$ occurs, we can bound
	\begin{align*}
		|\widehat{V}_n| \le \sum_{i=1}^n (a_i+1) \le 4 \sum_{i=1}^n D_i \le 4 (\log n)^{\frac{3}{\tau-1}}\sum_{i=1}^n \left( \frac{n}{i}\right)^{\frac{1}{\tau-1}}.
	\end{align*}
We can choose~$k > 0$ such that the right-hand side above is smaller than~$n^k$ when~$n$ is large enough. 
Then, with high probability as~$n \to \infty$, we have
	\[
		\mathrm{diam}(\widehat{G}_n) \le C_0 \log(|\widehat{V}_n|) \le C_0 k \log n,
	\]
	concluding the proof.
\end{proof}

To prove Lemma~\ref{lem_max_degree}, we will need the following.

\begin{lemma}\label{lem_flatten}
For any $\tau>2$ and any~$\epsilon > 0$, with high probability as~$n \to \infty$, the number of edges lost in erasing the configuration model is smaller than~$\epsilon n$.
\end{lemma}

\begin{proof}
	We assume that some arbitrary order is taken for the stubs of the configuration model. For each~$j \in \mathbb N$, we define
	\[L_n(j):=\mathbb P(\text{stub 1 of vertex 1 is lost due to erasing in~$G_n$} \mid D_1=j)\]
	(in the conditioning, we leave the stub-counts of all vertices~$i \neq 1$ unrevealed).

	Let~$R_n$ denote the number of edges deleted due to erasing. Since each edge removes two stubs, using symmetry, we can bound
\[
\frac{2\mathbb E R_n}{n}
	\leq \sum_{j\ge1}jp(j)L_n(j).
\]
	Equation~(9) in~\cite{MR2266448} gives~$L_n(j)\xrightarrow{n \to \infty} 0$ for every~$j$. Hence, using the monotone convergence theorem, we obtain
	\[
		\frac{\mathbb E R_n}{n} \xrightarrow{n \to \infty} 0.
	\]
	The result then follows from Markov's inequality.
\end{proof}

\begin{proof}[Proof of Lemma~\ref{lem_max_degree}]
	For~$\tau > 2$, this follows from Lemma~\ref{lem_event_An} (which guarantees that there are vertices whose stub-counts are powers of~$n$, with high probability) together with \cite[Equation 4.9]{MR3876906} (which tells us that the largest degree in the erased graph is with high probability asymptotic to its stub-count).

	For~$\tau=2$, label the vertices in non-increasing order of their stub-counts and assume that the stub-counts satisfy the event~$\mathcal A_n$ of~\eqref{eq_def_event_An}. We can then repeat the proof of Lemma~\ref{lem_event_Bn} (with the minor adjustment that an extra~$\log n$ term appears in the upper bound for~$L_n$) to show that there exists~$K > 0$ such that with high probability as~$n \to \infty$, every two distinct vertices~$i,j$ with~$ij \le n/(\log n)^K$ are neighbours in~$G_n$. Then, with high probability as~$n \to \infty$, the vertex with highest stub-count neighbours every vertex~$j$ with~$j \le n/(\log n)^K$.
\end{proof}

\subsection{Proofs of statements from Section~\ref{s_coex}}
As in Section~\ref{s_coex}, here we assume that~$\tau > 2$.

\begin{proof}[Proof of Lemma~\ref{lem_small_complements}]
	Choose~$M \in \mathbb N$ large enough that~$\sum_{k \ge M} kp(k) < \epsilon/2$ and then choose~$\delta>0$ with~$\delta < p(\{k:k\ge M\})/4$. By the strong law of large numbers, with high probability when~$n \to \infty$, the stub-count sequence satisfies
	\begin{equation}\label{eq_first_counts}
		\frac{1}{n}\sum_{i}D_i \mathds{1}\{D_i \ge M\} \le 2\sum_{k \ge M} kp(k) < \epsilon
	\end{equation}
	and
	\begin{equation}
		\label{eq_second_counts}
		\frac{1}{n} \sum_i \mathds{1}\{D_i \ge M\} > \frac12 p(\{k:k\ge M\}) > 2\delta.
	\end{equation}
	Assume these conditions are satisfied. Recall that~$(D_{(1)},\ldots, D_{(n)})$ are the stub-counts in decreasing order. Note that~\eqref{eq_second_counts} implies that~$D_{(1)},\ldots,D_{(\lfloor 2 \delta n \rfloor)} \ge M$. Given a set of vertices~$A$ with~$|A|\le \delta n$, we have
	\begin{equation*}
		e\big(A:A^\complement\big)\le \sum_{i \in A} D_i \le \sum_{i=1}^{\lfloor \delta n \rfloor} D_{(i)} \le \sum_{i=1}^{n} D_i \mathds{1}\{D_i \ge M\} < \epsilon n. \qedhere
	\end{equation*}
\end{proof}

\begin{proof}[Proof of Lemma~\ref{lem_large_complements}]
	We denote by~$e(A:B)$ the number of edges between~$A$ and~$B$ in~$G_n$, and~$e^{\rm full}(A:B)$ the same quantity in the unerased version of~$G_n$ (that is, the graph still containing parallel edges).

	By \cite[Lemma 12]{MR2957935}, with high probability as~$n \to \infty$, 
\[
	e^{\rm full}\big(A:A^{\,\complement}\big) \geq \frac{1}{100} \min\bigg\{ \sum_{i\in A} D_i,\; \sum_{i \in A^\complement} D_i \bigg\} \quad \text{for all }A \subseteq [n].
\]
with high probability.
Because our graph has $D_i\geq 3$ for every \(i \in V\),
\[
e^{\rm full}\big(A:A^{\,\complement}\big) \geq \frac{3}{100}\min\{|A|,|A^c|\}.
\]
	In particular, for any~$\gamma \in (0,1/2)$, with high probability as~$n \to \infty$ we have
\[
    \min_{A:\gamma n \leq |A| \leq (1-\gamma)n} e\big(A:A^{\,\complement}\big) \geq \frac{\gamma n}{50}.
\]
	The desired result now follows from Lemma~\ref{lem_flatten}, which tells us that with high probability, there are $o(n)$ edges differing between the unerased and the erased graph.
\end{proof}

\printbibliography

\end{document}